\documentclass[12pt]{article}

\usepackage{amsfonts,mathrsfs}
\usepackage{amssymb,amsmath,amsbsy,amsthm}
\usepackage{dsfont}
\usepackage{blkarray}
\usepackage{tikz}
\usepackage{tkz-euclide}
\usepackage{tikz-cd}
\usetikzlibrary{patterns}
\usepackage{ae}
\usepackage{bm}
\usepackage{graphicx}
\usepackage{float}
\usepackage{cite}
 \usepackage{indentfirst}
\usepackage{lineno}
\usepackage{titlesec}
\usepackage{enumitem}
\usepackage{color}
\usepackage{aliascnt}
\usepackage{varioref}
\usepackage{hyperref}
\hypersetup{colorlinks=true,linkcolor=cyan,anchorcolor=cyan,citecolor=red,CJKbookmarks=True,
,urlcolor=cyan}
\usepackage{cleveref}

\numberwithin{equation}{section}

\def\int{\mbox{\rm int}}

\def\And{\mbox{\rm ~and~}}

\def\({\mbox{\rm (}}\def\){\mbox{\rm )}}

\makeatletter

\newcommand{\Rmnum}[1]{\expandafter\@slowromancap\romannumeral #1@}
\makeatother
\newtheorem{theorem}{Theorem}[section]
\newaliascnt{lemma}{theorem}
\newtheorem{lemma}[lemma]{Lemma}
\aliascntresetthe{lemma}

\newaliascnt{proposition}{theorem}
\newtheorem{proposition}[proposition]{Proposition}
\aliascntresetthe{proposition}

\newaliascnt{fact}{theorem}

\aliascntresetthe{fact}

\newaliascnt{definition}{theorem}
\newtheorem{definition}[definition]{Definition}
\aliascntresetthe{definition}

\newaliascnt{conjecture}{theorem}

\aliascntresetthe{conjecture}

\newaliascnt{corollary}{theorem}
\newtheorem{corollary}[corollary]{Corollary}
\aliascntresetthe{corollary}

\newaliascnt{claim}{theorem}

\aliascntresetthe{claim}

\newaliascnt{problem}{theorem}

\aliascntresetthe{problem}

\newaliascnt{question}{theorem}

\aliascntresetthe{question}

\newaliascnt{remark}{theorem}

\aliascntresetthe{remark}

\newaliascnt{example}{theorem}
\newtheorem{example}[example]{Example}
\aliascntresetthe{example}

\newaliascnt{notation}{theorem}

\aliascntresetthe{notation}

\begin{document}
\begin{center}
{\Large\bf
Characteristic polynomials of semimatroids and their connections to matroids, hyperplane arrangements and graph colorings}\\[7pt]
\end{center}
\begin{center}
Houshan Fu\\[5pt]
School of Mathematics and Information Science\\
 Guangzhou University\\
Guangzhou 510006, Guangdong, P. R. China\\[5pt]
Email: fuhoushan@gzhu.edu.cn\\[15pt]
\end{center}
\begin{abstract}
We primarily investigate the properties of characteristic polynomials of semimatroids. In particular, we provide a combinatorial interpretation of their coefficients, generalizing the Whitney's Broken Circuit Theorem. We also prove that the unsigned coefficients of the characteristic polynomial form a unimodal and log-concave sequence, extending the Rota-Heron-Welsh Conjecture to semimatroids. Furthermore, we present convolution identities for the multiplicative characteristic and Tutte polynomials of semimatroids using the M\"obius conjugation. Finally, motivated by Kochol's work, we introduce assigning matroids to establish connections among semimatroids, hyperplane arrangements, and graph colorings,  with a particular focus on their characteristic polynomials.
\vspace{1ex}
\vspace{1ex}
\end{abstract}
\section{Introduction}\label{Sec1}
Chromatic polynomials or characteristic polynomials, and Tutte polynomials are the most important and extensively studied polynomial invariants in the fields of graph theory, matroid theory, hyperplane arrangements and semimatroids. The origin of the Tutte polynomial can be traced back to the chromatic polynomial, which is a graph polynomial arising from questions in graph coloring \cite{Tutte1954}. In order to solve the long-standing Four Colour Problem, the chromatic polynomial of planar graphs was first proposed by Birkhoff \cite{Birkhoff1912} in 1912, and later extended to general graphs by Whitney \cite{Whitney1932}. Subsequently, Whitney \cite{Whitney1935} introduced the concept of matroids to capture abstract properties of dependence common to graphs and vectors. In 1964, Rota \cite{Rota1964} defined the characteristic polynomial of matroids as an extension of chromatic polynomials. Further extending this work, Crapo \cite{Crapo1969} introduced the Tutte polynomial of a matroid, which serves as a generalization of the characteristic polynomial. 

In 2007, Ardila \cite{Ardila2007} extended the concepts of characteristic and Tutte polynomials to semimatroids. A semimatroid is an abstract structure that captures the properties of affine hyperplane arrangements, generalizing the concept of matroids. This was created by Wachs and Walker \cite{Wachs-Walker1986} in terms of the ``geometric lattice" of closed sets, and later developed a comprehensive theory in terms of subsets of a finite set by Ardila \cite{Ardila2007} (also discovered independently by Kawahara \cite{Kawahara2004}). Ardila's work has centered on analyzing the properties of the Tutte polynomials of semimatroids.
Specifically, Ardila established a deletion-contraction formula for the Tutte polynomials of semimatroids in \cite[Proposition 8.2]{Ardila2007}. He further showed that the coefficients $t^is^j$ of the Tutte polynomial of a semimatroid are nonnegative integers and correspond to the number of certain bases with internal activity $i$ and external activity $j$ in \cite[Theorem 9.5]{Ardila2007}, which mirrors the well-known results for matroids given by Crapo \cite{Crapo1969}. However, he overlooked the study of characteristic polynomials of semimatroids.

Our initial motivation is to address the previously overlooked study of characteristic polynomials of semimatroids. In \autoref{Sec3-2}, we characterize the characteristic polynomials of semimatroids through various methods, in analogy with matroids. The exploration of the coefficients of characteristic polynomials of matroids is a fascinating topic. Whitney \cite{Whitney1932,Whitney1932-1} pioneered this field by introducing a combinatorial interpretation for the coefficients of chromatic polynomials via the concept of broken circuits, now famously known as the Whitney's Broken Circuit Theorem. Three decades later, Rota \cite{Rota1964} extended this result to matroids, a discovery that Brylawski \cite{Brylawski1977} independently achieved. Inspired by their work, we provide a similar combinatorial interpretation for the unsigned coefficients of the characteristic polynomials of semimatroids in \autoref{Generalized-Broken-Circuit-Theorem}.

Furthermore, the investigation of the unimodality and log-concavity of the coefficients of  characteristic polynomials has become a central theme in both graph theory and matroid theory. More precisely, the unimodality of the sequence of unsigned coefficients of the characteristic polynomial of a loopless matroid was implicit in Rota \cite{Rota1971} and explicit in Heron \cite{Heron1972}, as a generalization of an earlier graph-theoretic conjecture of Read \cite{Read1968}. Subsequently, Welsh \cite{Welsh1976} conjectured that this sequence is log-concave. This is commonly known as the Rota-Heron-Welsh Conjecture and generalizes Hoggar's conjecture \cite{Hoggar1974} for chromatic polynomials. These conjectures have been confirmed by Huh et al. \cite{AHE2018,Huh2012,Huh-Katz2012}. Building on previous work, we show that the sequence of unsigned coefficients of the characteristic polynomial of a semimatroid is unimodal and log-concave in \autoref{Semi-Log-Concave}.

Another prominent subfield in graph theory and matroid theory focuses on a variety of convolution formulae associated with multiplicative chromatic polynomials, multiplicative characteristic polynomials and Tutte polynomials. Multiplicative identities for chromatic polynomials were first proposed by Tutte \cite{Tutte1967}. After more than two decades, Kook, Reiner and Stanton \cite{KRS1999} established a convolution formula for the Tutte polynomial of matroids using incidence algebra methods. This result was also independently discovered by Etienne and Las Vergnas \cite{EL1998}. Subsequently, Kung presented a number of convolution-multiplication identities for multiplicative characteristic polynomials and Tutte polynomials of  graphs and matroids in \cite{Kung2004,Kung2010}. Since then, this topic has been substantially refined and generalized \cite{BL2017, DFM2018, Reiner1999}. Recently, Wang \cite{Wang2015} introduced the concept of M\"obius conjugation of posets as a unified method to reprove previous convolution formulae, and first provided a convolution formula for multiplicative characteristic polynomials of hyperplane arrangements. A natural question arises: how can those convolution identities be extended to semimatroids?  In \autoref{Sec4}, we further explore this question using the M\"obius conjugation of posets.

Further motivation comes from the desire to establish connections among semimatroids, hyperplane arrangements and graph colorings. To this end, we introduce the concept of assigning matroids, which originates from Kochol's work. In 2022, Kochol \cite{Kochol2022} introduced assigning polynomials counting nowhere-zero chains in graphs---nonhomogeneous analogues of nowhere-zero flows (called $(A,b)$-flows in \cite{LY2006}). This is a generalization of the classical flow polynomial introduced by Tutte \cite{Tutte1954}. Subsequently, Kochol \cite{Kochol2024} extended this approach for a regular matroid to enumerate nowhere-zero chains associated with homomorphisms and zero-one assignings from its circuits to $\{0,1\}$. An {\em assigning matroid} is an ordered pair $(M,\alpha)$ consisting of an arbitrary matroid $M$ and its {\em assigning} $\alpha$ that is a mapping from circuits of $M$ to the set $\{0,1\}$. 

We show that every semimatroid can induce an assigning matroid in \autoref{Semimatroid-Inverse}. This allows us to use the assigning matroid as a bridge connecting semimatroids to hyperplane arrangements and graph colorings. For more details, refer to \autoref{Sec5} for hyperplane arrangements, and to \autoref{Sec6} for graph colorings.

The paper is organized as follows. In \autoref{Sec3}, we mainly introduce basic definitions regarding semimatroids and assigning matroids, and explore properties of the characteristic polynomials of semimatroids as well as establish a relationship between semimatroids and assigning matroids. \autoref{Sec4} is devoted to examining the convolution identities for the multiplicative characteristic and Tutte polynomials of semimatroids. Applying this to hyperplane arrangements, \autoref{Sec5} provides a convolution formula for their Tutte polynomials; and revisits the classification problem of parallel translations of a hyperplane arrangement based on relevant properties of  semimatroids and assigning matroids. In \autoref{Sec6}, our primary objective is to investigate applications of assigning matroids to assigning graphs, with a particular focus on graph colorings. 
\section{Semimatroids}\label{Sec3}
This section is concerned with two central goals. One is to explore the properties of characteristic polynomials of semimatroids. The other is to study the connection between semimatroids and assigning matroids.
\subsection{Characteristic polynomials of semimatroids}\label{Sec3-2}
In this subsection, we focus on studying the properties of characteristic polynomials of semimatroids. For this purpose, let us first review the necessary definitions and results connecting semimatroids that will be needed later. For further reading on this topic, we refer the reader to the literature \cite{Ardila2007}.

A {\em semimatroid} is a triple $(E,\mathcal{C},r_{\mathcal{C}})$ consisting of a finite set $E$, a nonempty simplicial complex $\mathcal{C}$ on $E$, and a function $r_\mathcal{C}:\mathcal{C}\to\mathbb{N}$, satisfying the following five properties:
\begin{itemize}[leftmargin=1.4cm]
\item[{\rm (SR1)}] If $X\in\mathcal{C}$, then $0\le r_{\mathcal{C}}(X)\le|X|$.
\item[{\rm (SR2)}] If $X,Y\in\mathcal{C}$ and $X\subseteq Y$, then $r_{\mathcal{C}}(X)\le r_{\mathcal{C}}(Y)$.
\item[{\rm (SR3)}] If $X,Y\in\mathcal{C}$ and $X\cup Y\in\mathcal{C}$, then 
\[
r_{\mathcal{C}}(X\cap Y)+r_{\mathcal{C}}(X\cup Y)\le r_{\mathcal{C}}(X)+r_{\mathcal{C}}(Y).
\]
\item[{\rm (SR4)}] If $X,Y\in\mathcal{C}$ and $r_{\mathcal{C}}(X)=r_{\mathcal{C}}(X\cap Y)$, then $X\cup Y\in\mathcal{C}$.
\item[{\rm (SR5)}] If $X,Y\in\mathcal{C}$ and $r_{\mathcal{C}}(X)<r_{\mathcal{C}}(Y)$, then $X\cup e\in\mathcal{C}$ for some $e\in Y-X$.
\end{itemize}
Every member in $\mathcal{C}$ is referred to as a {\em central set} of the semimatroid $(E,\mathcal{C},r_{\mathcal{C}})$. $E$, $\mathcal{C}$ and $r_{\mathcal{C}}$ are called the {\em ground set}, the {\em collection of central sets} and the {\em rank function} of $(E,\mathcal{C},r_{\mathcal{C}})$, respectively. We often write $(E,\mathcal{C},r_{\mathcal{C}})$ simply as $\mathcal{C}$ when its ground set and rank function are clear. In addition, all the maximal sets in $\mathcal{C}$ have the same rank, which is denoted by $r(\mathcal{C})$ and called the {\em rank}  of the semimatroid $\mathcal{C}$. A set $X\in\mathcal{C}$ is {\em independent} if $r_{\mathcal{C}}(X)=|X|$, and {\em dependent} otherwise. Denoted by $\mathscr{I}(\mathcal{C})$ the collection of all independent sets of $\mathcal{C}$. Moreover, a maximal independent set and a minimal dependent set of $\mathcal{C}$ are  referred to as a {\em basis} and a {\em circuit} in turn. In particular, an element $e\in\mathcal{C}$ is called a {\em bridge} if $e$ is contained in every basis, and called a {\em loop} of $\mathcal{C}$ if $e$ is a circuit (i.e., $r_\mathcal{C}(e)=0$). We use the notation $\mathscr{C}(\mathcal{C})$ to represent the set of circuits of $\mathcal{C}$. The {\em closure} of $X$ in $\mathcal{C}$ is defined as
\[
{\rm cl}_{\mathcal{C}}(X):=\big\{x\in E\mid X\cup e\in\mathcal{C},\;r_{\mathcal{C}}(X\cup e)=r_{\mathcal{C}}(X)\big\}.
\]
Particularly, a member $X\in\mathcal{C}$ is a {\em flat} of  the semimatroid $\mathcal{C}$ if ${\rm cl}_{\mathcal{C}}(X)=X$. Furthermore, the poset $\mathscr{L}(\mathcal{C})$, which consists of all flats of $\mathcal{C}$ ordered by set inclusion, forms a geometric semilattice, as shown in \cite[Theorem 6.4]{Ardila2007}. In fact, ${\rm cl}_{\mathcal{C}}(X)$ is the unique maximal element in $\mathcal{C}$ containing $X$ with rank $r_{\mathcal{C}}(X)$, as shown in the proof of {\rm (CLR1)} in \cite[Proposition 2.4]{Ardila2007}.

In this paper, we shall work with a fixed semimatroid $(E,\mathcal{C},r_{\mathcal{C}})$. Let $X\in\mathcal{C}$. The {\em restriction} $\mathcal{C}|X$ of $\mathcal{C}$ to $X$ or the {\em deletion} $\mathcal{C}\backslash (E-X)$ from $\mathcal{C}$, is a semimatroid $(X,\mathcal{C}|X,r_{\mathcal{C}|X})$ on the ground set $X$ with 
\begin{equation}\label{Deletion1}
\mathcal{C}|X=\mathcal{C}\backslash (E-X):=\{Y\subseteq X\mid Y\in\mathcal{C}\}
\end{equation}
and the rank function
\begin{equation}\label{Deletion-Rank-Semi}
r_{\mathcal{C}|X}(Y):=r_{\mathcal{C}}(Y),\quad\forall \;Y\in\mathcal{C}|X.
\end{equation}
The {\em contraction} $\mathcal{C}/X$ is a semimatroid $(E-X,\mathcal{C}/X,r_{\mathcal{C}/X})$ on the ground set $E-X$ with
\begin{equation}\label{Contraction1}
\mathcal{C}/X:=\{Y\subseteq E-X\mid Y\sqcup X\in\mathcal{C}\}
\end{equation}
and the rank function
\begin{equation}\label{Contraction-Rank-Semi}
r_{\mathcal{C}/X}(Y):=r_{\mathcal{C}}(Y\sqcup X)-r_{\mathcal{C}}(X),\quad\forall \;Y\in\mathcal{C}/X.
\end{equation}

The {\em Tutte polynomial} of $\mathcal{C}$ introduced by Ardila \cite[Definition 8.1]{Ardila2007}, is defined as
\[
T(\mathcal{C};t,s):=\sum_{X\in\mathcal{C}}(t-1)^{r(\mathcal{C})-r_{\mathcal{C}}(X)}(s-1)^{|X|-r_{\mathcal{C}}(X)}.
\]
In particular, define $T(\mathcal{C};t,s):=1$ when $E=\emptyset$. Ardila also defined the {\em characteristic polynomial} of $\mathcal{C}$ by
\[
\chi(\mathcal{C};t):=\sum_{X\in\mathcal{C}}(-1)^{|X|}t^{r(\mathcal{C})-r_{\mathcal{C}}(X)}.
\]
Likewise, define $\chi(\mathcal{C};t)=1$ for $E=\emptyset$. Obviously, $T(\mathcal{C};t,s)$ is a generalization of  $\chi(\mathcal{C};t)$ related by 
\[
\chi(\mathcal{C};t)=(-1)^{r(\mathcal{C})}T(\mathcal{C};1-t,0).
\]
Analogous to the Tutte polynomial of matroids, \cite[Proposition 8.2]{Ardila2007} demonstrated that the Tutte polynomial of semimatroids also satisfies the following deletion-contraction rule:
\begin{equation*}
T(\mathcal{C};t,s)=
\begin{cases}
sT(\mathcal{C}/e;t,s),&\text{if $e$ is a loop};\\
tT(\mathcal{C}\backslash e;t,s),&\text{if $e$ is a bridge};\\
T(\mathcal{C}\backslash e;t,s)+T(\mathcal{C}/e;t,s),&\text{if $e$ is neither a bridge nor a loop and $e\in\mathcal{C}$};\\
T(\mathcal{C}\backslash e;t,s),&\text{if $e\in E$ and $e\notin\mathcal{C}$}.
\end{cases}
\end{equation*}

We assume that readers are familiar with the basics of matroids. The matroid terminology follows Oxley's book \cite{Oxley2011}. Let $M$ be a matroid. We use notations $E(M)$, $\mathscr{L}(M)$, $\mathscr{I}(M)$, $\mathscr{C}(M)$ and $r_M$ to denote the ground set, the poset consisting of its flats, the set of all its independent sets, the set of all its circuits and its rank function, respectively.

It is well known that an alternative characterization of the characteristic polynomial $\chi(M,t)$ of a loopless matroid $M$ is associated with the M\"obius function of $\mathscr{L}(M)$ (see \cite{Rota1964}). Next, we aim to provide a similar formula for the characteristic polynomial of semimatroids. We still have some work to do before we achieve this.

Poset terminology can refer to Stanley's book \cite{Stanley2012}. Let $P$ be a poset. Define the {\em M\"obius function} $\mu_P=\mu:P\times P\to \mathbb{Z}$ as follows:
\[
\mu(X,Y)=\begin{cases}
1,&\text{if $X=Y\in P$};\\
-\sum_{X\le Z<Y}\mu(X,Z),&\text{if $X,Z,Y\in P$ and $X<Y$};\\
0,&\text{otherwise}.
\end{cases}
\]
For simplicity, we denote by $\mu(X):=-\sum_{Z<X}\mu(Z)$ if $X,Z\in P$. The following essential lemma provides the characterization of the M\"obius function $\mu$ of $\mathscr{L}(\mathcal{C})$.
 
\begin{lemma}\label{Mobius-Function}
Let $(E,\mathcal{C},r_\mathcal{C})$ be a semimatroid. For any $X,Y\in\mathscr{L}(\mathcal{C})$ with $X\subseteq Y$, let $\mathcal{C}(X,Y):=\big\{Z\in \mathcal{C}\mid X\subseteq Z\subseteq Y\And {\rm cl}_\mathcal{C}(Z)=Y\big\}$. Then
\[
\mu(X,Y)=\sum_{Z\in\mathcal{C}(X,Y)}(-1)^{|Z-X|}.
\]
\end{lemma}
\begin{proof}
The case that $X=Y$, is trivial. Suppose $X\subset Y$ in $\mathscr{L}(\mathcal{C})$. Let $\nu(X,Y):=\sum_{Z\in\mathcal{C}(X,Y)}(-1)^{|Z-X|}$. From \cite[Proposition 3.4]{Ardila2007}, we know that $\mathcal{C}$ is a simplicial complex. In other words, if $S\in \mathcal{C}$, then each subset of $S$ belongs to $\mathcal{C}$ as well. It follows that 
\[
\bigsqcup_{X\subseteq Z\subseteq Y\,{\rm in}\,\mathscr{L}(\mathcal{C})}\mathcal{C}(X,Z)=\big\{S\subseteq E\mid X\subseteq S\subseteq Y\big\},
\]
where the union is disjoint. Immediately, we have
\[
\sum_{X\subseteq Z\subseteq Y\,{\rm in}\,\mathscr{L}(\mathcal{C})}\nu(X,Z)=\sum_{X\subseteq S\subseteq Y}(-1)^{|S-X|}=0,
\]
which finishes the proof.
\end{proof}

Applying the deletion-contraction formula for the Tutte polynomial $T(\mathcal{C};t,s)$ to the relation $\chi(\mathcal{C};t)=(-1)^{r(\mathcal{C})}T(\mathcal{C};1-t,0)$, we directly derive that $\chi(\mathcal{C};t)$ satisfies the following deletion-contraction  recurrence: 
\begin{equation}\label{Semi-Characteristic-DCF}
\chi(\mathcal{C};t)=
\begin{cases}
(t-1)\chi(\mathcal{C}\backslash e;t),& \text{if $e$ is a bridge};\\
0,& \text{if $e$ is a loop};\\
\chi(\mathcal{C}\backslash e;t),& \text{if $e\notin\mathcal{C}$};\\
\chi(\mathcal{C}\backslash e;t)-\chi(\mathcal{C}/e;t),& \text{otherwise}.
\end{cases}
\end{equation}
According to this formula, we conclude that the characteristic polynomial $\chi(\mathcal{C};t)$ is the zero polynomial if $\mathcal{C}$ contains a loop. Excluding this case, the poset $\mathscr{L}(\mathcal{C})$ allows us to provide an alternative expression for characteristic polynomials of semimatroids (We do not know an explicit source for this formula). 
\begin{theorem}\label{Mobius-Characteristic-Semi}
Let $(E,\mathcal{C},r_{\mathcal{C}})$ be a semimatroid without loops. Then, the characteristic polynomial $\chi(\mathcal{C};t)$ can be expressed as 
\[
\chi(\mathcal{C};t)=\sum_{X\in\mathscr{L}(\mathcal{C})}\mu(X)t^{r(\mathcal{C})-r_{\mathcal{C}}(X)}.
\]
\end{theorem}
\begin{proof}
Clearly, $\emptyset\in\mathscr{L}(\mathcal{C})$.  It follows from \autoref{Mobius-Function} that
\[
\sum_{X\in\mathscr{L}(\mathcal{C})}\mu(X)t^{r(\mathcal{C})-r_{\mathcal{C}}(X)}= \sum_{X\in\mathscr{L}(\mathcal{C})}\Big(\sum_{S\in\mathcal{C}(\emptyset,X)}(-1)^{|S|}\Big)t^{r(\mathcal{C})-r_{\mathcal{C}}(X)}. 
\]
As every central subset $S\in\mathcal{C}$ corresponds to a unique flat ${\rm cl}_\mathcal{C}(S)\in\mathscr{L}(\mathcal{C})$, we have the following decomposition 
\[
\mathcal{C}=\bigsqcup_{X\in\mathscr{L}(\mathcal{C})}\mathcal{C}(\emptyset,X).
\]
This implies that 
\[
\sum_{X\in\mathscr{L}(\mathcal{C})}\mu(X)t^{r(\mathcal{C})-r_{\mathcal{C}}(X)}=\sum_{X\in\mathcal{C}}(-1)^{|X|}t^{r(\mathcal{C})-r_{\mathcal{C}}(X)}=\chi(\mathcal{C};t).
\]
This completes the proof.
\end{proof}
A natural question is whether there is a combinatorial interpretation for the coefficients of characteristic polynomials of semimatroids. To this end, let us review the Whitney's Broken Circuit Theorem for matroids, which was independently established by Brylawski \cite{Brylawski1977} and Rota \cite{Rota1964}. Let $M$ be a matroid and $\prec$ a linear ordering of $E(M)$. A {\em broken circuit} of $M$ is obtained from a circuit of $M$ by removing its minimal element with respect to $\prec$. 
\begin{proposition} [\cite{Brylawski1977,Rota1964}]\label{WBCT}
Let $M$ be a loopless matroid. Then, the unsigned coefficient $w_i(M)$ of $t^{r(M)-i}$ in the characteristic polynomial $\chi(M;t)$ equals the number of subsets of $E(M)$ that have size $i$ and do not contain broken circuits.
\end{proposition}
Extending this to semimatroids, the upcoming result provides a similar interpretation for the coefficients of characteristic polynomials of semimatroids. We could not find an explicit source for this, but the relevant concept of a {\em balanced broken circuit} of semimatroids was introduced in \cite{FZ2016}. Let $\prec$ be a linear ordering of $E$.  A subset $X$ of $E$ is called a {\em broken circuit} of $\mathcal{C}$ if $X$ is obtained by removing the minimal element from a circuit of $\mathcal{C}$ with respect to $\prec$. We sometimes express $\chi(\mathcal{C};t)$ in the form
\[
\chi(\mathcal{C};t)=\sum_{i=0}^{r(\mathcal{C})}(-1)^iw_i(\mathcal{C})t^{r(\mathcal{C})-i}.
\]
Each unsigned coefficient $w_i(\mathcal{C})$ in $\chi(\mathcal{C};t)$ is referred to as the {\em $i$-th unsigned Whitney number of the first kind} of $\mathcal{C}$. 
\begin{theorem}[Broken Circuit Theorem]\label{Generalized-Broken-Circuit-Theorem}
Let $(E,\mathcal{C},r_{\mathcal{C}})$ be a semimatroid and $\prec$ a linear ordering of $E$. Then, every unsigned coefficient $w_i(\mathcal{C})$ equals the number of central subsets in $\mathcal{C}$ that have size $i$ and do not contain broken  circuits with respect to $\prec$.
\end{theorem}
\begin{proof}
Assume $E=\{e_1\prec e_2\prec\cdots\prec e_m\}$. We prove the result by induction on $m$. If $m=0$,  there is nothing to prove. Suppose $m\ge 1$ and the result holds for smaller values of $m$. If $e_m$ is a loop of $\mathcal{C}$, it is trivial via \eqref{Semi-Characteristic-DCF}. If $e_m\notin\mathcal{C}$, from $\chi(\mathcal{C};t)=\chi(\mathcal{C}\backslash e_m;t)$ in \eqref{Semi-Characteristic-DCF}, the result holds in this case by the inductive hypothesis. 

If $e_m\in\mathcal{C}$ is a non-loop of $\mathcal{C}$, by the inductive hypothesis, the result holds for both semimatroids $\mathcal{C}\backslash e_m$ and $\mathcal{C}/e_m$. By \eqref{Semi-Characteristic-DCF},  for $i=1,2,\ldots,r(\mathcal{C})$, we have 
\[
w_0(\mathcal{C})=w_0(\mathcal{C} \backslash e_m)\quad\And\quad w_i(\mathcal{C},\alpha)=w_i(\mathcal{C}\backslash e_m)+w_{i-1}(\mathcal{C}/e_m).
\]
Therefore, to complete the induction step, it is sufficient to check that the above relations are true. But this is a consequence of the following two simple facts: (a) given a central subset $X\subseteq E-e_m$, $X$ contains no broken circuits of $\mathcal{C}$ if and only if $X\in\mathcal{C}\backslash e_m$ does not contain broken  circuits of $\mathcal{C}\backslash e_m$ with respect to $\prec$; (b) given a central subset $X\subseteq E$ with $e_m\in X$, $X$ contains no broken circuits of $\mathcal{C}$ if and only if $X-e_m$ does not contain broken circuits of $\mathcal{C}/e_m$ with respect to $\prec$.
\end{proof}

It is worth pointing out that the coefficients $w_i(\mathcal{C})$ in \autoref{Generalized-Broken-Circuit-Theorem} are independent of the choice of linear ordering of $E$, and \autoref{Generalized-Broken-Circuit-Theorem} aligns with \autoref{WBCT} when $\mathcal{C}$ is a matroid.  The next corollary is an immediate result of \autoref{Generalized-Broken-Circuit-Theorem}.
\begin{corollary}\label{Value-1}
Let $(E,\mathcal{C},r_\mathcal{C})$ be a semimatroid and a linear ordering $\prec$ of $E$. Then the value $(-1)^{r(\mathcal{C})}\chi(\mathcal{C};-1)$ equals the number of central subsets in $\mathcal{C}$ that do not contain broken circuits with respect to $\prec$.
\end{corollary}

As noted in  \cite[Proposition 3.1]{Ardila2007}, every semimatroid can be used to construct a matroid. Specifically, by extending the rank function $r_{\mathcal{C}}$ from $\mathcal{C}$ to $2^E$ via defining 
\begin{equation}\label{Rank-Extending}
r(X):={\rm max}\big\{r_{\mathcal{C}}(Y)\mid Y\subseteq X,Y\in\mathcal{C}\big\},\quad\forall\,X\in 2^E,
\end{equation}
we obtain a matroid $(E,r)$ with the rank function $r$, denoted by $M_{\mathcal{C}}$. The following consequence further compares the coefficients of the characteristic polynomials $\chi(M_\mathcal{C};t)$ and $\chi(\mathcal{C};t)$.
\begin{corollary}\label{Comparison}
Let $(E,\mathcal{C},r_\mathcal{C})$ be a semimatroid without loops. Then 
\[
w_i(M_\mathcal{C})\le w_i(\mathcal{C}) \quad \mbox{ for }\quad i=0,1,\ldots,r(\mathcal{C}).
\]
\end{corollary}
\begin{proof}
Let $\prec$ be a linear ordering of $E$. We denote by $\mathscr{C}'(M_\mathcal{C})$ and $\mathscr{C}'(\mathcal{C})$ the sets of broken circuits of $M_\mathcal{C}$ and $\mathcal{C}$ with respect to $\prec$, respectively. Let $\mathscr{C}''(M_\mathcal{C})$ and $\mathscr{C}''(\mathcal{C})$ represent the sets of  central subsets of $M_\mathcal{C}$ and $\mathcal{C}$ that contain no broken circuits with respect to $\prec$, respectively. 

By \autoref{Generalized-Broken-Circuit-Theorem}, the proof  reduces to showing that every member of $\mathscr{C}''(M_\mathcal{C})$ belongs to $\mathscr{C}''(\mathcal{C})$. Given an element $X\in\mathscr{C}''(M_\mathcal{C})$. Since $X$ does not contain broken circuits of $M_\mathcal{C}$ with respect to $\prec$, $X$ contains no circuits of $M_\mathcal{C}$. It follows that $X$ is an independent set of $M_\mathcal{C}$. This implies that $X$ is also independent set of $\mathcal{C}$ and $X\in\mathcal{C}$. Note that  $\mathscr{C}(\mathcal{C})\subseteq\mathscr{C}(M_\mathcal{C})$. This means $\mathscr{C}'(\mathcal{C})\subseteq\mathscr{C}'(M_\mathcal{C})$. It deduces that $X$ does not contain broken circuits of $\mathcal{C}$ with respect to $\prec$. Thus $\mathscr{C}''(M_\mathcal{C})\subseteq \mathscr{C}''(\mathcal{C})$. This completes the proof.
\end{proof}

Let $M$ be a loopless matroid. A well known result states that the coefficients of the characteristic polynomial $\chi(M;t)$ strictly alternate in sign. Following \autoref{Comparison}, we present a direct consequence of this below. 
\begin{corollary}\label{Alternate-Sign-Coefficient}
Let $(E,\mathcal{C},r_\mathcal{C})$ be a semimatroid without loops. Then the coefficients in $\chi(\mathcal{C};t)$ are nonzero and alternate in sign with $w_0(\mathcal{C})=1$.
\end{corollary}

In the study of semimatroids, there are many equivalent ways to define a semimatroid. We can define a semimatroid in terms of a matroid $\tilde{N}$ on the ground set $E\sqcup p$ and a distinguished element $p$ as well. The pair $(\tilde{N},p)$ is called a {\em pointed matroid}. \cite[Theorem 5.4]{Ardila2007} showed that an arbitrary pointed matroid $(\tilde{N},p)$ on the ground set  $E\sqcup p$ defines a unique semimatroid $(E,\mathcal{C},r_{\mathcal{C}})$ such that $\mathcal{C}=\big\{X\subseteq E\mid p\notin {\rm cl}_{\tilde{N}}(X)\big\}$ and $r_{\mathcal{C}}$ is the restriction of $r_{\tilde{N}}$ to $\mathcal{C}$. This was implicit in the work of Wachs and Walker \cite{Wachs-Walker1986}. Based on this, Ardila further obtained a close link between the characteristic polynomials of semimatroids and matroids, in analogy to the relationship between the characteristic polynomials of affine arrangements and their conings (see \cite[Proposition 2.51]{Orlik-Terao1992} and \cite{Stanley2007}).
\begin{proposition}[\cite{Ardila2007}, Proposition 8.7]\label{Relation-Characteristic-Matroid-Semimatroid}
For any semimatroid $\mathcal{C}$,
\[
\chi(\tilde{N};t)=(t-1)\chi(\mathcal{C};t),
\]
where a matroid $\tilde{N}$ together with its loopless element $p$ forms a pointed matroid $(\tilde{N},p)$ on the ground set  $E\sqcup p$ such that $\mathcal{C}=\big\{X\subseteq E\mid p\notin {\rm cl}_{\tilde{N}}(X)\big\}$ and $r_{\mathcal{C}}$ is the restriction of $r_{\tilde{N}}$ to $\mathcal{C}$.
\end{proposition}
Next, let us discuss the unimodality and log-concavity of coefficients of characteristic polynomials of semimatroids. A well known result is presented as follows.
\begin{proposition}[\cite{AHE2018}, Theorem 9.9]\label{Huh}
The unsigned coefficients of the reduced characteristic polynomial $\bar{\chi}(M;t)$ of  a matroid $M$ form a log-concave sequence, where $\bar{\chi}(M;t):=\chi(M;t)/(t-1)$. Moreover, the unsigned coefficients of the characteristic polynomial $\chi(M;t)$ of  $M$ form a log-concave sequence.
\end{proposition}

From \autoref{Relation-Characteristic-Matroid-Semimatroid}, we can view the characteristic polynomial of a semimatroid as the reduced characteristic polynomial of the corresponding pointed matroid. Besides, a basic fact is that a log-concave sequence must be a unimodal sequence.  It follows from  \autoref{Huh} that the unsigned coefficients of the characteristic polynomial $\chi(\mathcal{C};t)$ form a log-concave sequence, and hence a unimodal sequence. 
\begin{corollary}\label{Semi-Log-Concave}
Let $(E,\mathcal{C},r_{\mathcal{C}})$ be a semimatroid without loops. Then, the sequence $w_0(\mathcal{C})$, $w_1(\mathcal{C})$, $\ldots$, $w_{r(\mathcal{C})}(\mathcal{C})$ satisfies the following properties:
\begin{itemize}
\item The sequence $w_i(\mathcal{C})$ is unimodal, i.e., there is an index $0\le k\le r(\mathcal{C})$ such that 
\[
w_0(\mathcal{C})\le\cdots\le w_{k-1}(\mathcal{C})\le w_k(\mathcal{C})\ge w_{k+1}(\mathcal{C})\ge\cdots\ge w_{r(\mathcal{C})}(\mathcal{C}).
\]
\item The sequence $w_i(\mathcal{C})$ is log-concave, i.e., for $k=1,\ldots, r(\mathcal{C})-1$,
 \[
w_{k-1}(\mathcal{C})w_{k+1}(\mathcal{C})\le  w_k(\mathcal{C})^2.
 \]
\end{itemize}
\end{corollary}
\subsection{Semimatroids associated with assigning matroids}\label{Sec3-1}
This subsection is devoted to investigating the relationship between semimatroids and assigning matroids.  An {\em assigning} $\alpha$ of $M$ is a mapping from $\mathscr{C}(M)$ to the set $\{0,1\}$.  Write $\alpha\equiv0$ if $\alpha(C)=0$ for all circuits $C$ of $M$, or if $M$ contains no circuits. An {\em assigning matroid} $(M,\alpha)$ is an ordered pair consisting of a matroid $M$ and its assigning $\alpha$. Let 
\[
\mathscr{C}(M,\alpha):=\big\{C\in\mathscr{C}(M)\mid \alpha(C)=0\big\},
\] 
called the {\em compatible circuit set} of $(M,\alpha)$.  A circuit in $\mathscr{C}(M,\alpha)$, and more broadly any submatroid of $M$ or subset of $E(M)$  whose circuits are in $\mathscr{C}(M,\alpha)$, is {\em $\alpha$-compatible} (or simply {\em compatible}) and {\em incompatible} otherwise. We consistently use the notation $\mathcal{C}(M,\alpha)$ to denote the set of all compatible subsets of $(M,\alpha)$, i.e.,
\[
\mathcal{C}(M,\alpha):=\big\{X\subseteq E(M)\mid X \text{ is compatible}\big\}.
\]
It is clear that if $X$ is compatible, then every subset of $X$ is compatible. Thus,  $\mathcal{C}(M,\alpha)$ is a simplicial complex. Furthermore, we define the {\em rank function} $r_{M,\alpha}$ of $(M,\alpha)$ as the restriction of the rank function $r_M$ to the set $\mathcal{C}(M,\alpha)$.

Extending the characteristic and Tutte polynomials of matroids to assigning matroids, we propose the compatible characteristic and Tutte polynomials of assigning matroids. Indeed, the compatible characteristic and Tutte polynomials of $(M,\alpha)$ in \autoref{Compatible-Characteristic-Def} align with the classical characteristic and Tutte polynomials of $M$ whenever $M$ is compatible, respectively.
\begin{definition}\label{Compatible-Characteristic-Def}
{\rm Let $(M,\alpha)$ be an assigning matroid. 
\begin{itemize}
\item The {\em compatible characteristic polynomial}  of $(M,\alpha)$ is defined as 
\[
\chi(M,\alpha;t):=\sum_{X\in\mathcal{C}(M,\alpha)}(-1)^{|X|}t^{r(M)-r_M(X)}.
\] 
Particularly,  define $\chi(M,\alpha;t)=1$ for the empty matroid $M$ with $E(M)=\emptyset$.
\item The {\em compatible Tutte polynomial} of $(M,\alpha)$ is defined by
\[
T(M,\alpha;t,s):=\sum_{X\in\mathcal{C}(M,\alpha)}(t-1)^{r(M)-r_M(X)}(s-1)^{|X|-r_M(X)}.
\]
Particularly, we define $T(M,\alpha;t,s)=1$ for the empty matroid $M$.
\end{itemize}
}
\end{definition}
In enumerative problems concerning colorings (tensions) and flows, the compatible characteristic polynomial is crucial, analogous to the chromatic (tension) polynomial and the flow polynomial in \cite{Tutte1954}. Specifically, the compatible characteristic polynomial for a graphic matroid associated with its admissible assigning agrees with the assigning polynomial in Theorem 1 of \cite{Kochol2024}, which enumerates certain nowhere-zero chains and $(A,f)$-tensions, corresponding to the potential differences of the $(A,f)$-coloring. Dually, \cite[Theorem 1.3]{FRW2025}, \cite[Theorem 1]{Kochol2022} and \cite[Section 5]{Kochol2024} showed that the compatible characteristic polynomial for a cographic matroid related to its admissible assigning corresponds to the assigning polynomial that counts nowhere-zero chains in graphs---nonhomogeneous analogues of nowhere-zero flows.

We now start by showing how we construct an assigning matroid from a semimatroid. It is clear that the semimatroid $\mathcal{C}$ naturally induces an assigning $\alpha_{\mathcal{C}}$ of $M_{\mathcal{C}}$ such that for any circuit $C$ of $M_{\mathcal{C}}$,
\[
\alpha_\mathcal{C}(C):=
\begin{cases}
0,&\text{if $C\in\mathcal{C}$};\\
1,&\text{otherwise}.
\end{cases}
\]
In other words, the semimatroid $\mathcal{C}$ defines an assigning matroid $(M_\mathcal{C},\alpha_{\mathcal{C}})$. On the other hand, the following theorem demonstrates how to recover the given semimatroid $\mathcal{C}$ from the assigning matroid $(M_\mathcal{C},\alpha_{\mathcal{C}})$.  
\begin{theorem}\label{Semimatroid-Inverse}
Let $(E,\mathcal{C},r_{\mathcal{C}})$  be a semimatroid. Then, the triple $\big(E,\mathcal{C}(M_\mathcal{C},\alpha_\mathcal{C}),r_{M_\mathcal{C},\alpha_\mathcal{C}}\big)$ is a semimatroid and coincides with the semimatroid $(E,\mathcal{C},r_{\mathcal{C}})$.
\end{theorem}
\begin{proof}
The proof reduces to proving that \[\big(E,\mathcal{C}(M_\mathcal{C},\alpha_\mathcal{C}),r_{M_\mathcal{C},\alpha_\mathcal{C}}\big)=(E,\mathcal{C},r_{\mathcal{C}}).\]
The definition of $r_{M_{\mathcal{C}}}$ in \eqref{Rank-Extending} directly implies that the restriction of $r_{M_{\mathcal{C}}}$ to $\mathcal{C}$ coincides with $r_{\mathcal{C}}$. Since $r_{M_{\mathcal{C}},\alpha_{\mathcal{C}}}$ is also the restriction of $r_{M_{\mathcal{C}}}$ to  $\mathcal{C}(M_\mathcal{C},\alpha_{\mathcal{C}})$, it suffices to verify $\mathcal{C}=\mathcal{C}(M_\mathcal{C},\alpha_{\mathcal{C}})$.

The construction of the assigning $\alpha_{\mathcal{C}}$ directly leads to $\mathcal{C}\subseteq\mathcal{C}(M_\mathcal{C},\alpha_{\mathcal{C}})$. For the opposite inclusion,  arguing by negation, suppose there exists $X\in\mathcal{C}(M_\mathcal{C},\alpha_{\mathcal{C}})-\mathcal{C}$. Applying \eqref{Rank-Extending} again, we derive $\mathscr{I}(M_\mathcal{C})=\mathscr{I}(\mathcal{C})$. Therefore, we can choose an independent set $I$ in $\mathcal{C}$ such that $I\subseteq X$ and  $r_{M_\mathcal{C}}(X)=r_{\mathcal{C}}(I)=|I|$. As $X\notin\mathcal{C}$, there exist $e\in X-I$ and a circuit $C$ of $M_\mathcal{C}$ for which $e\in C\subseteq I\sqcup e\subseteq X$ and $C\notin\mathcal{C}$. Additionally, we can easily observe from \eqref{Rank-Extending} and the definition of $\alpha_\mathcal{C}$ that every circuit of $\mathcal{C}$ is precisely a compatible circuit of the assigning matroid $(M_\mathcal{C},\alpha_\mathcal{C})$, and vice versa, i.e., $\mathscr{C}(\mathcal{C})=\mathscr{C}(M_\mathcal{C},\alpha_{\mathcal{C}})$. It follows $C\notin \mathscr{C}(M_\mathcal{C},\alpha_{\mathcal{C}})$. This deduces $X\notin\mathcal{C}(M_\mathcal{C},\alpha_{\mathcal{C}})$, a contradiction. So, $\mathcal{C}=\mathcal{C}(M_\mathcal{C},\alpha_{\mathcal{C}})$ holds. 
\end{proof}
Following \autoref{Semimatroid-Inverse}, a semimatroid is essentially an assigning matroid. However, an assigning matroid is not necessarily a semimatroid, as illustrated in the next example. 
\begin{example}\label{Eg2}
{\rm
Consider the uniform matroid $U_{2,4}$ with rank $2$ and ground set $E=\{e_i\mid i=1,2,3,4\}$. Clearly, $\mathscr{C}(U_{2,4})=\{C_i=E-e_i\mid i=1,2,3,4\}$. \autoref{Table1} shows that $\mathcal{C}(U_{2,4},\alpha)$ is a semimatroid if at most one $C_i$ has $\alpha(C_i)=0$ or if $\alpha\equiv0$, but it is not a semimatroid if two or three $C_i$ have $\alpha(C_i)=0$.
\begin{table}[H]
\centering
\begin{tabular}{|c|c|c|c|}
\hline
Assigning Matroid $(U_{2,4},\alpha)$  & $\mathcal{C}(U_{2,4},\alpha)$ &Semimatroid \\ \hline
$\alpha(C_i)=1,i=1,2,3,4$ & $2^E-\mathscr{C}(U_{2,4})-E$&\text{Yes}  \\ \hline
$\alpha(C_1)=0,\;\alpha(C_i)=1,i=2,3,4$ & $2^E-\mathscr{C}(U_{2,4})-E\sqcup C_1$&\text{Yes}  \\ \hline 
$\alpha(C_2)=0,\;\alpha(C_i)=1,i=1,3,4$&$2^E-\mathscr{C}(U_{2,4})-E\sqcup C_2$&\text{Yes}  \\ \hline
$\alpha(C_3)=0,\;\alpha(C_i)=1,i=1,2,4$& $2^E-\mathscr{C}(U_{2,4})-E\sqcup C_3$&\text{Yes} \\ \hline
$\alpha(C_4)=0,\;\alpha(C_i)=1,i=1,2,3$& $2^E-\mathscr{C}(U_{2,4})-E\sqcup C_4$&\text{Yes} \\ \hline
$\alpha(C_1)=\alpha(C_2)=0,\;\alpha(C_3)=\alpha(C_4)=1$& $2^E-C_3-C_4-E$&\text{No} \\ \hline
$\alpha(C_1)=\alpha(C_3)=0,\;\alpha(C_2)=\alpha(C_4)=1$& $2^E-C_2-C_4-E$&\text{No}  \\ \hline
$\alpha(C_1)=\alpha(C_4)=0,\;\alpha(C_2)=\alpha(C_3)=1$& $2^E-C_2-C_3-E$&\text{No} \\ \hline
$\alpha(C_2)=\alpha(C_3)=0,\;\alpha(C_1)=\alpha(C_4)=1$& $2^E-C_1-C_4-E$ &\text{No} \\ \hline
$\alpha(C_2)=\alpha(C_4)=0,\;\alpha(C_1)=\alpha(C_3)=1$& $2^E-C_1-C_3-E$&\text{No}  \\ \hline
$\alpha(C_3)=\alpha(C_4)=0,\;\alpha(C_1)=\alpha(C_2)=1$& $2^E-C_1-C_2-E$&\text{No}  \\ \hline
$\alpha(C_i)=0,i=1,2,3,\;\alpha(C_4)=1$& $2^E-C_4-E$&\text{No} \\ \hline
$\alpha(C_i)=0,i=1,2,4,\;\alpha(C_3)=1$& $2^E-C_3-E$ &\text{No} \\ \hline
$\alpha(C_i)=0,i=1,3,4,\;\alpha(C_2)=1$& $2^E-C_2-E$&\text{No} \\ \hline
$\alpha(C_i)=0,i=2,3,4,\;\alpha(C_1)=1$& $2^E-C_1-E$ &\text{No} \\ \hline
$\alpha(C_i)=0,i=1,2,3,4$& $2^E$&\text{Yes}  \\ \hline 
\end{tabular}
\caption{Assigning matroids $(U_{2,4},\alpha)$ and the corresponding semimatroids $\mathcal{C}(U_{2,4},\alpha)$.}
\label{Table1}
\end{table}
}
\end{example}
Based on \autoref{Semimatroid-Inverse}, it is straightforward to see that the Tutte and characteristic polynomials of a semimatroid $\mathcal{C}$ align with  the compatible Tutte and characteristic polynomials of the assigning matroid $(M_{\mathcal{C}},\alpha_\mathcal{C})$, respectively. 
\begin{corollary}\label{Relation-Compatible-Semimatroid1}
Let $(E,\mathcal{C},r_{\mathcal{C}})$ be a semimatroid. Then
\begin{itemize}
\item[{\rm(a)}] The Tutte polynomial $T(\mathcal{C};t,s)$ of  the semimatroid $\mathcal{C}$ coincides with the compatible Tutte polynomial $T(M_{\mathcal{C}},\alpha_{\mathcal{C}};t,s)$ of the assigning matroid $(M_{\mathcal{C}},\alpha_\mathcal{C})$.
\item[{\rm(b)}] The characteristic polynomial $\chi(\mathcal{C};t)$ of  the semimatroid $\mathcal{C}$ coincides with the compatible characteristic polynomial $\chi(M_{\mathcal{C}},\alpha_{\mathcal{C}};t)$ of the assigning matroid $(M_{\mathcal{C}},\alpha_\mathcal{C})$.
\end{itemize}
\end{corollary}
According to \autoref{Relation-Compatible-Semimatroid1}, the properties of the characteristic and Tutte polynomials of semimatroids $\mathcal{C}$  are readily applicable to the compatible characteristic and Tutte polynomials of the corresponding assigning matroids $(M_\mathcal{C},\alpha_\mathcal{C})$. 
\section{Convolution formulae}\label{Sec4}
This subsection is devoted to examining convolution identities for multiplicative characteristic polynomials and Tutte polynomials of semimatroids. A key technique in our work is the Möbius conjugation of posets, first created  by Wang \cite{Wang2015}.

We start by reviewing the M\"obius conjugation. Let $P$ be a locally finite poset and $R$ be a ring with identity. An {\em interval } $[x,y]$ in $P$ is the set of all elements $z\in P$ such that $ x\le z\le y$. Define ${\rm Int}(P)$ as the interval space of $P$, and $\mathcal{I}(P,R):=\big\{\beta:{\rm Int}(P)\to R\big\}$ as the incidence algebra of $P$ whose multiplication structure is given by the convolution product. Specifically, for any $\beta,\gamma\in \mathcal{I}(P,R)$ and $x\le y$ in $P$,
\begin{equation}\label{Mobius-Product-Def}
\beta*\gamma(x,y)=\sum_{x\le z\le y}\beta(x,z)\gamma(z,y).
\end{equation}
Further define the {\em M\"obius conjugation} $\mu^*:R^P\to \mathcal{I}(P,R)$ to be
\begin{equation}\label{Mobius-Conjugation-Def}
\mu^*(f):=\mu*\delta(f)*\zeta,\quad f\in R^P,
\end{equation}
where $\zeta$, the convolution inverse of the M\"obius function $\mu$ of $P$, is given by $\zeta(x,y)=1$ for any $x\le y$ in $P$; and the mapping $\delta:R^P\to \mathcal{I}(P,R)$ is defined by that $\delta(f)(x,y)$ equals $f(x)$ if $x=y$ and $0$ otherwise for all $f\in R^P$ and $x\le y$ in $P$. Wang showed that the M\"obius conjugation is a ring homomorphism in \cite[Theorem 1.1]{Wang2015}.
\begin{proposition}[\cite{Wang2015}, Theorem 1.1]\label{Mobius-Conjugation}
With above settings, the M\"obius conjugation $\mu^*$ is a ring homomorphism, i.e.,
\[
\mu^*(fg)=\mu^*(f)*\mu^*(g),\quad\forall f,g\in R^P.
\]
\end{proposition}
Below we describe some properties of their closure. Although these properties are similar to those of matroids, we could not find an explicit source.
\begin{proposition}\label{Contraction-Deletion-Closure}
Let $(E,\mathcal{C},r_{\mathcal{C}})$ be a semimatroid and $X\in\mathcal{C}$. Then
\begin{itemize}
\item [{\rm(a)}] ${\rm cl}_{\mathcal{C}|X}(Y)={\rm cl}_{\mathcal{C}}(Y)\cap X$ for  every $Y\in\mathcal{C}|X$.
\item [{\rm(b)}] ${\rm cl}_{\mathcal{C}/X}(Y)={\rm cl}_{\mathcal{C}}(Y\sqcup X)-X$ for  every $Y\in\mathcal{C}/X$.
\end{itemize}
\end{proposition}
\begin{proof}
Part (a) is an immediate consequence of the fact that the rank function $r_{\mathcal{C}|X}$ is the restriction of $r_{\mathcal{C}}$ to the set $\mathcal{C}|X$. Analogously, part (b) follows from that $r_{\mathcal{C}/X}(Y)=r_{\mathcal{C}}(Y\sqcup X)-r_{\mathcal{C}}(X)$.  This completes the proof.
\end{proof}
Furthermore, as shown in \cite[Theorem 6.4]{Ardila2007}, the poset $\mathscr{L}(\mathcal{C})$ of flats of a semimatroid $(E,\mathcal{C},r_{\mathcal{C}})$ is a geometric semilattice. By artificially adding an additional maximal element $\hat{1}$ to $\mathscr{L}(\mathcal{C})$, the poset $\mathscr{L}(\mathcal{C})\sqcup\hat{1}$ forms a geometric lattice, denoted by $\hat{\mathscr{L}}(\mathcal{C})$. We often use the notation $\hat{0}$ to denote the unique minimal element in $\mathscr{L}(C)$ and $\hat{\mathscr{L}}(C)$. From now on, we assume that $r_{\mathcal{C}}(\hat{1})=\infty$ and $t^{-\infty}=0$. Following this, the characteristic polynomial $\chi(\mathcal{C};t)\in \mathbb{Z}[t]$ of $\mathcal{C}$ can be written in the form 
\[
\chi(\mathcal{C},t)=\sum_{X\in\hat{\mathscr{L}}(\mathcal{C})}\mu(X)t^{r(\mathcal{C})-r_{\mathcal{C}}(X)}.
\]
We proceed to establish several poset isomorphisms concerning $\mathscr{L}(\mathcal{C})$ and $\hat{\mathscr{L}}(\mathcal{C})$.
\begin{lemma}\label{Cong}
Let $(E,\mathcal{C},r_{\mathcal{C}})$ be a semimatroid. Then
\begin{itemize}
\item [{\rm(a)}] For any $X\in\hat{\mathscr{L}}(\mathcal{C})$, $[\hat{0},X]$ is $\mathscr{L}(\mathcal{C}|X)$ if $X\in\mathscr{L}(\mathcal{C})$ and $\hat{\mathscr{L}}(\mathcal{C})$ otherwise.
\item [{\rm(b)}] For any $X\in\hat{\mathscr{L}}(\mathcal{C})$, $[X,\hat{1}]\cong\mathscr{L}(\mathcal{C}/X)\sqcup\hat{1}$.
\end{itemize}
\end{lemma}
\begin{proof}
When $X\in\mathscr{L}(\mathcal{C})$ and $Y\subseteq X$,  we have ${\rm cl}_{\mathcal{C}}(Y)\subseteq X$. It follows from part (a) of \autoref{Contraction-Deletion-Closure} that ${\rm cl}_{\mathcal{C}|X}(Y)={\rm cl}_{\mathcal{C}}(Y)$. Therefore, \[\mathscr{L}(\mathcal{C}|X)=\big\{{\rm cl}_{\mathcal{C}|X}(Y)\mid Y\in\mathcal{C}|X\big\}=[\hat{0},X].\] 
The case $X=\hat{1}$ is trivial in part (a).

For part (b), applying part (b) of \autoref{Contraction-Deletion-Closure}, we arrive at \[\mathscr{L}(\mathcal{C}/X)=\big\{{\rm cl}_{\mathcal{C}}(Y\sqcup X)-X\mid Y\in\mathcal{C}/X\big\}.\] This leads to a natural order-preserving bijection from $\mathscr{L}(\mathcal{C}/X)\sqcup\hat{1}$ to $[X,\hat{1}]$, which maps $\hat{1}$ to $\hat{1}$ and each element $Y\in \mathscr{L}(\mathcal{C}/X)$ to $Y\sqcup X\in[X,\hat{1}]$. Hence, the isomorphism $[X,\hat{1}]\cong\mathscr{L}(\mathcal{C}/X)\sqcup\hat{1}$ holds.
\end{proof}
Now, let us return to the following convolution formula for the multiplicative characteristic polynomial $\chi(\mathcal{C};ts)$ of $\mathcal{C}$, which agrees with the convolution formula for the multiplicative characteristic polynomial of matroids in \cite[Theorem 4]{Kung2004}. However, \autoref{Semi-CCF} holds for semimatroids, a broader class of objects that includes matroids but is not limited to them. 
\begin{theorem}\label{Semi-CCF}
Let $(E,\mathcal{C},r_{\mathcal{C}})$ be a semimatroid without loops. Then
\[
\chi(\mathcal{C};ts)=\sum_{X\in\mathscr{L}(C)}t^{r(\mathcal{C})-r_{\mathcal{C}}(X)}\chi(\mathcal{C}|X;t)\chi(\mathcal{C}/X;s).
\]
\end{theorem}
\begin{proof}
We start by defining two mappings $f,g:\hat{\mathscr{L}}(\mathcal{C})\to\mathbb{Z}[t,s]$ such that $f(X)=t^{r(\mathcal{C})-r_\mathcal{C}(X)}$ and $g(X)=s^{r(\mathcal{C})-r_\mathcal{C}(X)}$. It follows from \eqref{Mobius-Product-Def} and \eqref{Mobius-Conjugation-Def} that for any $X\in\hat{\mathscr{L}}(\mathcal{C})$, we have
\begin{align*}
\mu^*(f)(\hat{0},X)&=[\mu*\delta(f)*\zeta](\hat{0},X)\\
&=\sum_{\hat{0}\le Y\le X}[\mu*\delta(f)](\hat{0},Y)\\
&=\sum_{\hat{0}\le Z\le Y\le X}\mu(\hat{0},Z)\delta(f)(Z,Y).
\end{align*}
Combining  that $\delta(f)(Z,Y)$ equals $f(Z)$ if $Z=Y$ and $0$ otherwise, we derive 
\[
\mu^*(f)(\hat{0},X)=\sum_{\hat{0}\le Z\le X}\mu(\hat{0},Z)t^{r(\mathcal{C})-r_\mathcal{C}(Z)}.
\]
Together with part (a) of \autoref{Cong} and the assumption $t^{r(\mathcal{C})-r_{\mathcal{C}}(\hat{1})}=0$, we further deduce
\begin{equation}\label{Key-Equation}
\mu^*(f)(\hat{0},X)=t^{r(\mathcal{C})-r_{\mathcal{C}}(X)}\chi(\mathcal{C}|X;t).
\end{equation}
As with $\mu^*(f)(\hat{0},X)$, applying $[X,\hat{1}]\cong\mathscr{L}(\mathcal{C}/X)\sqcup\hat{1}$ in \autoref{Cong}, we can achieve 
\[
\mu^*(g)(X,\hat{1})=\sum_{X\le Z\le \hat{1}}\mu(X,Z)s^{r(\mathcal{C})-r_\mathcal{C}(Z)}=\chi(\mathcal{C}/X;s).
\]
It follows from \autoref{Mobius-Conjugation} that
\begin{align*}
\chi(\mathcal{C};ts)&=[\mu^*(f)*\mu^*(g)](\hat{0},\hat{1})\\
&=\sum_{X\in\hat{\mathscr{L}}(\mathcal{C})}\mu^*(f)(\hat{0},X)\mu^*(g)(X,\hat{1})\\
&=\sum_{X\in\hat{\mathscr{L}}(\mathcal{C})}t^{r(\mathcal{C})-r_{\mathcal{C}}(X)}\chi(\mathcal{C}|X;t)\chi(\mathcal{C}/X;s).
\end{align*}
Applying $t^{r(\mathcal{C})-r_{\mathcal{C}}(\hat{1})}=0$ again, we obtain that $\chi(\mathcal{C};ts)$ can be written in the form
\[
\chi(\mathcal{C};ts)=\sum_{X\in\mathscr{L}(\mathcal{C})}t^{r(\mathcal{C})-r_{\mathcal{C}}(X)}\chi(\mathcal{C}|X;t)\chi(\mathcal{C}/X;s),
\]
which completes the proof.
\end{proof}

We continue to explore a convolution formula for the Tutte polynomial of semimatroids. Let $\hat{e}$ be an element  that is not contained in $E$ and $\hat{E}:=E\sqcup\hat{e}$. Similarly, we denote by $\hat{\mathcal{C}}$ the new set $\mathcal{C}\sqcup\hat{E}$, obtained from $\mathcal{C}$ by artificially adding an extra element $\hat{e}$. Likewise, we assume  $r_{\mathcal{C}}(\hat{E})=\infty$ and $t^{-\infty}=0$. Immediately, $T(\mathcal{C};t,s)$ can be expressed as
\[
T(\mathcal{C};t,s)=\sum_{X\in\hat{\mathcal{C}}}(t-1)^{r(\mathcal{C})-r_{\mathcal{C}}(X)}(s-1)^{|X|-r_{\mathcal{C}}(X)}.
\]
\begin{theorem}\label{Semi-TCF}
Let $(E,\mathcal{C},r_{\mathcal{C}})$ be a semimatroid. Then
\[
T(\mathcal{C};t,s)=\sum_{X\in\mathcal{C}}T(\mathcal{C}|X;0,s)T(\mathcal{C}/X;t,0).
\]
\end{theorem}
\begin{proof}
First note that for any element $X\in\mathcal{C}$, 
\[
T(\mathcal{C}|X;t,s)=\sum_{Y\subseteq X}(t-1)^{r_\mathcal{C}(X)-r_{\mathcal{C}}(Y)}(s-1)^{|Y|-r_{\mathcal{C}}(Y)}
\]
by \eqref{Deletion1} and \eqref{Deletion-Rank-Semi}, and
\[
T(\mathcal{C}/X;t,s)=\sum_{X\subseteq Y, \,Y\in\mathcal{C}}(t-1)^{r(\mathcal{C})-r_{\mathcal{C}}(Y)}(s-1)^{|Y-X|-r_{\mathcal{C}}(Y)+r_{\mathcal{C}}(X)}
\]
via \eqref{Contraction1} and \eqref{Contraction-Rank-Semi}.

To write the Tutte polynomial as the M\"obius conjugation, we now consider the poset $(\hat{\mathcal{C}},\subseteq)$ and define $f,g:\hat{\mathcal{C}}\to\mathbb{Z}[t,s]$ such that
\[
f(X)=(1-s)^{|X|-r_{\mathcal{C}}(X)}\quad\And\quad g(X)=(1-t)^{r(\mathcal{C})-r_{\mathcal{C}}(X)},\quad\forall\;X\in\hat{\mathcal{C}}.
\]
Since $\mathcal{C}$ is a simplicial complex, for any $X\subseteq Y$ in $\mathcal{C}$, the interval $[X,Y]$ is contained in $\mathcal{C}$ and is isomorphic to the Boolean lattice $[\emptyset,Y-X]$.  Consequently, the M\"obius function of $\hat{\mathcal{C}}$ satisfies $\mu(X,Y)=(-1)^{|Y-X|}$ for any $X\subseteq Y$ in $\mathcal{C}$. We adopt the same procedure as in the proof of  \eqref{Key-Equation}. Together with the assumptions that $r_{\mathcal{C}}(\hat{E})=\infty$ and $t^{-\infty}=0$, we can derive
\[
\mu^*(f)(\emptyset,X)=(-1)^{r(\mathcal{C}|X)}T(\mathcal{C}|X;0,s),
\]
\[
\mu^*(g)(X,\hat{E})=(-1)^{r(\mathcal{C}/X)}T(\mathcal{C}/X;t,0),
\]
\[
\mu^*(fg)(\emptyset,\hat{E})=(-1)^{r(\mathcal{C})}T(\mathcal{C};t,s).
\]
Combining $r(\mathcal{C})=r(\mathcal{C}|X)+r(\mathcal{C}/X)$ for every $X\in\mathcal{C}$ and applying \autoref{Mobius-Conjugation}, we deduce
\begin{align*}
T(\mathcal{C};t,s)&=(-1)^{r(\mathcal{C})}\mu^*(fg)(\emptyset,\hat{E})\\
&=(-1)^{r(\mathcal{C})}\sum_{X\in\hat{\mathcal{C}}}\mu^*(f)(\emptyset,X)\mu^*(g)(X,\hat{E})\\
&=\sum_{X\in\mathcal{C}}T(\mathcal{C}|X;0,s)T(\mathcal{C}/X;t,0).
\end{align*}
This completes the proof.
\end{proof}
The above convolution formula for the Tutte polynomial of semimatroids, using the following lemma, can be further simplified to the form presented in the subsequent corollary.
\begin{lemma}\label{Semi-Contraction-Loop}
Let $(E,\mathcal{C},r_{\mathcal{C}})$ be a semimatroid and $X\in\mathcal{C}$. The semimatroid $\mathcal{C}/X$ has no loops if and only if $X\in\mathscr{L}(\mathcal{C})$.
\end{lemma}
\begin{proof}
Assume $X\in\mathscr{L}(\mathcal{C})$. Arguing by contradiction, suppose $\mathcal{C}/X$ has a loop $e$. Then $e\notin X$ and $X\sqcup e\in \mathcal{C}$. Furthermore, we have
\begin{equation}\label{Rank-Equi}
0=r_{\mathcal{C}/X}(e)=r_\mathcal{C}(X\sqcup e)-r_\mathcal{C}(X).
\end{equation}
This implies $e\in {\rm cl}_\mathcal{C}(X)=X$, a contradiction. Therefore, $\mathcal{C}/X$ does not contain loops.

Conversely, suppose $X\notin\mathscr{L}(\mathcal{C})$. Then ${\rm cl}_\mathcal{C}(X)-X$ is nonempty. Taking an element $e\in {\rm cl}_\mathcal{C}(X)-X$, the equation \eqref{Rank-Equi} holds in this case. This means that $e$ is a loop of $\mathcal{C}/X$, a contraction. Hence, $X\in\mathscr{L}(\mathcal{C})$, which finishes the proof.
\end{proof}

\begin{corollary}\label{Semi-TCF1}
Let $(E,\mathcal{C},r_{\mathcal{C}})$ be a semimatroid. Then 
\[
T(\mathcal{C};t,s)=\sum_{X\in\mathscr{L}(\mathcal{C})}T(\mathcal{C}|X;0,s)T(\mathcal{C}/X;t,0).
\]
Moreover, $T(\mathcal{C};t,s)$ can be expressed as
\[
T(\mathcal{C};t,s)=\sum_{X\text{ is a cyclic flat of $\mathcal{C}$}}T(\mathcal{C}|X;0,s)T(\mathcal{C}/X;t,0),
\]
where a flat $X$ of $\mathcal{C}$ is a cyclic flat if it does not have bridges.
\end{corollary}
\begin{proof}
From the deletion-contraction formula for $T(\mathcal{C};t,s)$, we directly derive that $T(\mathcal{C};t,0)=0$ if  $\mathcal{C}$ has a loop; and $T(\mathcal{C};0,s)=0$ if  $\mathcal{C}$ has a bridge. Combining \autoref{Semi-TCF} and \autoref{Semi-Contraction-Loop}, we obtain the first formula in \autoref{Semi-TCF1}. From the definition of cyclic flats, we further deduce the second formula in \autoref{Semi-TCF1}. 
\end{proof}
\section{Hyperplane arrangements}\label{Sec5}
This section primarily focuses on two aspects. One is to present a convolution formula for the Tutte polynomial of hyperplane arrangements. The other is to explore the classification problem of parallel translations of an arbitrary linear arrangement from the perspective of assigning matroids and semimatroids.
\subsection{Convolution formulae}\label{Sec5-1}
We begin with general hyperplane arrangements, and then specialize to those of parallel translations of a hyperplane arrangement. For more general setting and background on hyperplane arrangement can be found in the literature \cite{Stanley2007}. 

Let $\mathbb{F}$ be a field. A {\em hyperplane arrangement} $\mathcal{A}$ in the vector space $\mathbb{F}^n$ is a collection of (affine) hyperplanes in $\mathbb{F}^n$, indexed as follows
\[
H_e: \bm\alpha_e\cdot \bm x=a_e,\quad e\in E.
\]
The {\em complement} of $\mathcal{A}$ is defined to be $M(\mathcal{A}):=\mathbb{F}^n-\bigcup_{H_e\in \mathcal{A}}H_e$. 
For every subset $X\subseteq E$, we denote by $\mathcal{A}_X$ the collection of hyperplanes in $\mathcal{A}$ indexed by $X$, and by $\cap\mathcal{A}_X$ the intersection $\bigcap_{e\in X}H_e$ if it is nonempty. Let $\mathcal{C}(\mathcal{A},E)$ be the family of subsets $X\subseteq E$ such that the intersection $\cap\mathcal{A}_X$ is nonempty, i.e.,
\[
\mathcal{C}(\mathcal{A},E):=\{X\subseteq E\mid\cap\mathcal{A}_X\ne\emptyset\}.
\]
Every element in $\mathcal{C}(\mathcal{A},E)$ is referred to as a {\em central subset} of $E$. Thus, every member $X$ in $\mathcal{C}(\mathcal{A},E)$ precisely gives rise to a {\em central subarrangement} $\mathcal{A}_X$ of $\mathcal{A}$, and vice versa. The {\em rank function} $r_{\mathcal{A}}$ for central subarrangements $\mathcal{A}_X$ of $\mathcal{A}$ associated with the central subsets $X$, is defined as
\[
r_{\mathcal{A}}(\mathcal{A}_X):=n-\dim(\cap\mathcal{A}_X)={\rm rank}\big\{(\bm\alpha_e,a_e)\mid e\in X\big\}.
\]
Especially, the {\em rank} $r(\mathcal{A})$ of $\mathcal{A}$ equals $n-\dim\big({\rm span}\{\bm\alpha_e:e\in E\}\big)$. Moreover, we call the poset
\[
\mathscr{L}(\mathcal{A}):=\big\{\cap\mathcal{A}_X\mid X\in\mathcal{C}(\mathcal{A},E)\big\}
\]
the {\em intersection semilattice} of $\mathcal{A}$, whose partial order is the reverse of set inclusion and the unique minimal element is the whole space $\bigcap_{H\in\emptyset}H:=\mathbb{F}^n$. Members of $\mathscr{L}(\mathcal{A})$ are known as {\em flats}. The {\em characteristic polynomial} $\chi(\mathcal{A};t)$ of $\mathcal{A}$ is defined as
\begin{equation*}\label{Characteristic-Polynomial}
\chi(\mathcal{A};t):=\sum_{X\in\mathcal{C}(\mathcal{A},E)}(-1)^{|X|}t^{{\rm dim}(\cap\mathcal{A}_X)}=\sum\limits_{X\in \mathscr{L}(\mathcal{A})}\mu(X)t^{{\rm dim}(X)}.
\end{equation*}
In particular, we define $\chi(\mathcal{A};t)=1$ when $\mathcal{A}$ is the empty arrangement.

Let $X\in \mathscr{L}(\mathcal{A})$. The {\em localization} $\mathcal{A}^{X}$ of $\mathcal{A}$ at $X$ is given by
\[
\mathcal{A}^{X}:=\{H\in\mathcal{A}\mid X\subseteq H\}.
\]
To better study convolution formulae for the Tutte polynomial, we adopt the following way to define the restriction $\mathcal{A}|X$, which differs from its usual definition ``$\mathcal{A}|X=\{H\cap X\ne\emptyset\mid  H\in \mathcal{A}-\mathcal{A}^X\}$". The {\em restriction} $\mathcal{A}|X$ of $\mathcal{A}$ on $X$ is a hyperplane arrangement in $X$ defined by
\[
\mathcal{A}|X:=\{H\cap X\ne\emptyset\mid  H\in \mathcal{A}\}.
\]
We need to be especially careful that $\mathcal{A}|X$ possibly contain $X$. In this case, $X$ is called a {\em degenerate hyperplane}; and $\mathcal{A}|X$ is said to be {\em degenerate}. Consequently, its characteristic polynomial is defined by $\chi(\mathcal{A}|X;t):=0$.

Furthermore, we define the {\em rank function} $r_{\mathcal{C}(\mathcal{A},E)}$ of $\mathcal{C}(\mathcal{A},E)$ by
\[
r_{\mathcal{C}(\mathcal{A},E)}(X):=r_{\mathcal{A}}(\mathcal{A}_X),\quad\forall\, X\in\mathcal{C}(\mathcal{A},E).
\]
As noted in \cite[Proposition 2.2]{Ardila2007}, the triple of a hyperplane arrangement, the set of its central subarrangements and rank function forms a semimatroid, i.e., the triple $\big(E,\mathcal{C}(\mathcal{A},E),r_{\mathcal{C}(\mathcal{A},E)}\big)$ is a semimatroid. Notice that the geometric semilattice $\mathscr{L}\big(\mathcal{C}(\mathcal{A},E)\big)$ of $\mathcal{C}(\mathcal{A},E)$ consists of those central subsets $X$ in $\mathcal{C}(\mathcal{A},E)$ such that no other central subsets with rank $r_{\mathcal{C}(\mathcal{A},E)}(X)$ contain $X$. Immediately, the intersection semilattice $\mathscr{L}(\mathcal{A})$ can be more straightforwardly expressed in the form
\[
\mathscr{L}(\mathcal{A})=\big\{\cap\mathcal{A}_X\mid X\in\mathscr{L}\big(\mathcal{C}(\mathcal{A},E)\big)\big\}.
\]
Naturally, $\mathscr{L}(\mathcal{A})$ is isomorphic to $\mathscr{L}\big(\mathcal{C}(\mathcal{A},E)\big)$. It follows that the characteristic polynomial of $\mathcal{A}$ and the corresponding semimatroid $\mathcal{C}(\mathcal{A},E)$ are closely related by the equation:
\begin{equation}\label{Semimatroid-Arrangement}
\chi(\mathcal{A};t)=t^{n-r(\mathcal{A})}\chi\big(\mathcal{C}(\mathcal{A},E);t\big).
\end{equation}
Based on this relation, when \autoref{Semi-CCF} is restricted to the specific semimatroid arising from an arbitrary hyperplane arrangement, we directly obtain a convolution formula for the multiplicative characteristic polynomial of hyperplane arrangements, which was initially discovered by Wang \cite[Theorem 2.1]{Wang2015}. 
\begin{proposition}[\cite{Wang2015}, Theorem 2.1]
Let $\mathcal{A}$ be a hyperplane arrangement in $\mathbb{F}^n$. Then
\[
\chi(\mathcal{A};ts)=\sum_{X\in \mathscr{L}(\mathcal{A})}\chi(\mathcal{A}^X;t)\chi(\mathcal{A}|X;s).
\]
\end{proposition}

As a generalization of the characteristic polynomial $\chi(\mathcal{A};t)$, the {\em Tutte polynomial} $T(\mathcal{A};t,s)$ of $\mathcal{A}$ is defined to be
\[
T(\mathcal{A};t,s):=\sum_{X\in\mathcal{C}(\mathcal{A},E)}(t-1)^{r(\mathcal{A})
-r_{\mathcal{A}}(\mathcal{A}_X)}(s-1)^{|X|-r_{\mathcal{A}}(\mathcal{A}_X)}.
\]
In particular, we define $T(\mathcal{A};t,s)=1$ when $\mathcal{A}$ is the empty arrangement. Immediately, the specialization to hyperplane arrangements of  \autoref{Semi-TCF} is \autoref{HCF}.
\begin{corollary}\label{HCF}
Let $\mathcal{A}$ be a hyperplane arrangement in $\mathbb{F}^n$. Then
\begin{align*}
T(\mathcal{A};t,s)&=\sum_{X\in\mathcal{C}(\mathcal{A},E)}T(\mathcal{A}_X;0,s)T(\mathcal{A}|\cap\mathcal{A}_X;t,0)\\
&=\sum_{X\in\mathscr{L}(\mathcal{C}(\mathcal{A},E))}T(\mathcal{A}_X;0,s)T(\mathcal{A}|\cap\mathcal{A}_X;t,0).
\end{align*}
\end{corollary}

\subsection{Parallel translations}\label{Sec5-2}
In this subsection, we shall revisit the classification problem of parallel translations of hyperplane arrangements from the perspective of semimatroids and assigning matroids. From now on, we shall work with a fixed linear arrangement $\mathcal{A}_{\bm o}:=\{H_e:\bm\alpha_e\cdot\bm x=0\mid e\in E\}$ in $\mathbb{F}^n$. For every vector $\bm a=(a_e)_{e\in E}\in\mathbb{F}^{|E|}$, it naturally leads to a hyperplane arrangement 
\[
\mathcal{A}_{\bm a}:=\{H_{e,a_e}:\bm \alpha_e\cdot\bm x=a_e\mid e\in E\},
\]
in $\mathbb{F}^n$, which is referred to as a {\em parallel translation} of $\mathcal{A}_{\bm o}$. Conversely, we also refer to the linear arrangement $\mathcal{A}_{\bm o}$ as the {\em centralization} of an arbitrary hyperplane arrangement $\mathcal{A}_{\bm a}$ in $\mathbb{F}^n$.

It is well known that the linear correlation among normal vectors of an arbitrary central arrangement naturally determines a matroid. Thus, the semimatroid $\big(E,\mathcal{C}(\mathcal{A}_{\bm o},E),r_{\mathcal{C}(\mathcal{A}_{\bm o},E)}\big)$ is precisely a matroid given by the linear arrangement $\mathcal{A}_{\bm o}$, denoted by $M_{\mathcal{A}_{\bm o}}$. 

From a geometric perspective, a subset $C\subseteq E$ is a circuit of $M_{\mathcal{A}_{\bm o}}$ if and only if  the intersection  $\cap\mathcal{A}_{\bm o,C}$ is nonempty, and removing any element $e\in C$ results in $\cap\mathcal{A}_{\bm o,C-e}=\cap\mathcal{A}_{\bm o,C}$. Hence, we sometimes refer to a circuit of $M_{\mathcal{A}_{\bm o}}$ as a {\em circuit} of $\mathcal{A}_{\bm o}$ as well. Naturally, the concept of a circuit can be extended to general hyperplane arrangements that may not be central. An {\em affine circuit} of $\mathcal{A}_{\bm a}$ is a subset $C\subseteq E$ such that $\cap\mathcal{A}_{\bm a,C}\ne\emptyset$ and $\cap\mathcal{A}_{\bm a,C-e}=\cap\mathcal{A}_{\bm a,C}$ for all $e\in C$. We denote the set of all affine circuits of $\mathcal{A}_{\bm a}$ by $\mathscr{C}(\mathcal{A}_{\bm a})$. In fact, every affine circuit of $\mathcal{A}_{\bm a}$ agrees with a circuit of the corresponding semimatroid $\big(E,\mathcal{C}(\mathcal{A}_{\bm a},E),r_{\mathcal{C}(\mathcal{A}_{\bm a},E)}\big)$.

From the linear algebra viewpoint, a subset $C\subseteq E$ is an affine circuit of $\mathcal{A}_{\bm a}$ if and only if the vectors $(\bm\alpha_e,a_e)$ indexed by $C$ form a minimal linearly dependent set. Immediately, when restricted to the linear arrangement $\mathcal{A}_{\bm o}$, there exists a unique vector $\bm c_C=(c_e)_{e\in E}\in\mathbb{F}^{|E|}$ (up to a nonzero scalar multiple) associated with a circuit $C$ of $\mathcal{A}_{\bm o}$ such that
\begin{equation}\label{Circuit-Vector}
\sum_{e\in E}c_e\bm\alpha_e=\bm0\quad\And\quad c_e\ne 0\text{ if and only if } e\in C.
\end{equation}
This vector  $\bm c_C$ is referred to as a {\em circuit vector} of $C$. Then these hyperplanes with the circuit vectors as normal vectors form a linear arrangement in $\mathbb{F}^{|E|}$, called  the {\em discriminantal arrangement} or {\em derived arrangement} of $\mathcal{A}_{\bm o}$ and denoted by $\delta\mathcal{A}_{\bm o}$. More specifically, the discriminantal arrangement $\delta\mathcal{A}_{\bm o}$ is defined by
\begin{equation}\label{Derived-Arrangement}
\delta\mathcal{A}_{\bm o}:=\big\{H_C:\bm c_C\cdot \bm x=0\mid C\in\mathscr{C}(\mathcal{A}_{\bm o})\big\}.
\end{equation}

In 1989, Manin and Schechtman \cite{MS1989} investigated parallel translations of a linear arrangement in general position by introduced  the discriminantal arrangement (also called Manin--Schechtman arrangement). Subsequently, Falk \cite{Falk1994} demonstrated that neither combinatorial nor topological structure of the discriminantal arrangement is independent of the original arrangement. In 1997, Bayer and Brandt \cite{BB1997} studied the discriminantal arrangements without general position assumptions on the original arrangement. Recently, Oxley and Wang \cite{OW2019} observed that the vector matroid associated to the discriminantal arrangement  is nothing but the {\em derived matroid}, which was first proposed by Rota at the Bowdoin College Summer 1971 NSF Conference on Combinatorics to investigate ``dependencies among dependencies" of matroids. Not long after, Chen, Fu and Wang \cite{CFW2021} used the discriminantal arrangement $\delta\mathcal{A}_{\bm o}$ of  $\mathcal{A}_{\bm o}$ to classify intersection semilattices of all parallel translations $\mathcal{A}_{\bm a}$. For more related work, please see \cite{FW2023}.

\begin{proposition}[\cite{CFW2021}, Theorem 1.1]\label{Classification1}
Let $\bm a,\bm a'\in\mathbb{F}^{|E|}$. If there exists a flat $X\in\mathscr{L}(\delta\mathcal{A}_{\bm o})$ such that $\bm a,\bm a'\in M(\delta\mathcal{A}_{\bm o}/X)$, then $\mathscr{L}(\mathcal{A}_{\bm a})\cong\mathscr{L}(\mathcal{A}_{\bm a}')$.
\end{proposition}
\autoref{Classification1} may not provide a complete characterization for the isomorphic classification of the intersection semilattices of all parallel translations $\mathcal{A}_{\bm a}$, because $\mathscr{L}(\mathcal{A}_{\bm a})$ is possibly isomorphic to $\mathscr{L}(\mathcal{A}_{\bm a}')$ even if $\bm a$ and $\bm a'$ come from different flats of $\delta\mathcal{A}_{\bm o}$. However, the method in \autoref{Classification1} is a very good way to characterize distinct semimatroids $\big(E,\mathcal{C}(\mathcal{A}_{\bm a},E),r_{\mathcal{C}(\mathcal{A}_{\bm a},E)}\big)$ arising from  parallel translations $\mathcal{A}_{\bm a}$. This will be presented in the upcoming theorem. 

It is easy to see that every parallel translation $\mathcal{A}_{\bm a}$ automatically leads to an assigning $\alpha_{\bm a}$ of the matroid $M_{\mathcal{A}_{\bm o}}$ defined by that for every circuit $C\in \mathscr{C}(M_{\mathcal{A}_{\bm o}})=\mathscr{C}(\mathcal{A}_{\bm o})$,
\begin{equation}\label{Assigning-Arrangement}
\alpha_{\bm a}(C):=
\begin{cases}
0,&\text{if $C\in\mathscr{C}(\mathcal{A}_{\bm a})$};\\
1,&\text{otherwise}.
\end{cases}
\end{equation}
Recall from \autoref{Semimatroid-Inverse} that the triple $\big(E,\mathcal{C}(M_{\mathcal{A}_{\bm o}},\alpha_{\bm a}),r_{M_{\mathcal{A}_{\bm o}},\alpha_{\bm a}}\big)$ determined by the assigning matroid $(M_{\mathcal{A}_{\bm o}},\alpha_{\bm a})$ is a semimatroid and aligns with the semimatroid $\big(E,\mathcal{C}(\mathcal{A}_{\bm a},E),r_{\mathcal{C}(\mathcal{A}_{\bm a},E}\big)$. Following this, we establish the next result: every flat of the discriminantal arrangement $\delta\mathcal{A}_{\bm o}$ gives rise to a unique semimatroid (assigning matroid), which implies \autoref{Classification1} directly.

\begin{theorem}\label{Classification2}
Let $\bm a,\bm a'\in\mathbb{F}^{|E|}$.  Then the following statements are equivalent:
\begin{itemize}
\item[\rm 1.] $\big(E,\mathcal{C}(\mathcal{A}_{\bm a},E),r_{\mathcal{C}(\mathcal{A}_{\bm a},E)}\big)=\big(E,\mathcal{C}(\mathcal{A}_{\bm a'},E),r_{\mathcal{C}(\mathcal{A}_{\bm a'},E)}\big)$;
\item[\rm 2.] $(M_{\mathcal{A}_{\bm o}},\alpha_{\bm a})=(M_{\mathcal{A}_{\bm o}},\alpha_{\bm a'})$;
\item[\rm 3.] $\alpha_{\bm a}=\alpha_{\bm a'}$;
\item[\rm 4.] $\bm a$ and $\bm a'$ belong to $M(\delta\mathcal{A}_{\bm o}/X)$ for some flat $X$ of $\delta\mathcal{A}_{\bm o}$. 
\end{itemize}
Moreover, all parallel translations $\mathcal{A}_{\bm a}$ of $\mathcal{A}_{\bm o}$ correspond exactly to $|\mathscr{L}(\delta\mathcal{A}_{\bm o})|$ distinct semimatroids $\big(E,\mathcal{C}(\mathcal{A}_{\bm a},E),r_{\mathcal{C}(\mathcal{A}_{\bm a},E)}\big)$, and to $|\mathscr{L}(\delta\mathcal{A}_{\bm o})|$ distinct assigning matroids $(M_{\mathcal{A}_{\bm o}},\alpha_{\bm a})$.
\end{theorem}
\begin{proof}
It is an immediate result of \autoref{Semimatroid-Inverse} that statement $1$ is equivalent to statement $2$. Trivially, statement $2$ is equivalent to statement $3$. The remaining task is to prove the equivalence of statements $3$ and $4$. According to the definition of the discriminantal arrangement $\delta\mathcal{A}_{\bm o}$ in \eqref{Derived-Arrangement}, this reduces to verifying that $\alpha_{\bm a}(C)=0$ if and only if $\bm a\in H_C$ for every circuit $C\in\mathscr{C}(M_{\mathcal{A}_{\bm o}})$.  Let $C$ be a circuit of $M_{\mathcal{A}_{\bm o}}$. From the definition of $\alpha_{\bm a}$ in \eqref{Assigning-Arrangement}, $\alpha_{\bm a}(C)=0$ if and only if $C$ is an affine circuit of $\mathcal{A}_{\bm a}$. From the linear algebra, this is equivalent to both sets $\{\bm\alpha_e\mid e\in C\}$ and $\big\{(\bm\alpha_e,a_e)\mid e\in C\big\}$ being minimal linearly dependent sets. Using \eqref{Circuit-Vector}, this is further equivalent to $\bm a\in H_C$. Thus, the equivalence of statements $3$ and $4$ holds.

Moreover, let $\bm a,\bm a'\in\mathbb{F}^{|E|}$. The equivalence of statements 1 and 4 illustrates that $\big(E,\mathcal{C}(\mathcal{A}_{\bm a},E),r_{\mathcal{C}(\mathcal{A}_{\bm a},E)}\big)=\big(E,\mathcal{C}(\mathcal{A}_{\bm a'},E),r_{\mathcal{C}(\mathcal{A}_{\bm a'},E)}\big)$ if and only if $\bm a$ and $\bm a'$ come from the same stratum $M(\delta\mathcal{A}_{\bm o}/X)$ for some flat $X$ of $\delta\mathcal{A}_{\bm o}$. Together with the  basic decomposition formula 
\[
\mathbb{F}^{|E|}=\bigsqcup_{X\in\mathscr{L}(\delta\mathcal{A}_{\bm o})}M(\delta\mathcal{A}_{\bm o}/X),
\]
we conclude that all parallel translations $\mathcal{A}_{\bm a}$ yield precisely $|\mathscr{L}(\delta\mathcal{A}_{\bm o})|$ distinct semimatroids $\big(E,\mathcal{C}(\mathcal{A}_{\bm a},E),r_{\mathcal{C}(\mathcal{A}_{\bm a},E)}\big)$. Similarly, statements 2 and 4 can derive that all parallel translations $\mathcal{A}_{\bm a}$ define $|\mathscr{L}(\delta\mathcal{A}_{\bm o})|$ distinct assigning matroids $(M_{\mathcal{A}_{\bm o}},\alpha_{\bm a})$.
\end{proof}

Furthermore, we obtain a unified order-preserving relation from assignings $\alpha_{\bm a}$ to compatible characteristic polynomials $\chi(M_{\mathcal{A}_{\bm o}},\alpha_{\bm a};t)$ when both are naturally ordered. 
\begin{corollary}\label{Comparison1}
Let $\bm a,\bm a'\in\mathbb{F}^{|E|}$. 
\begin{itemize}
\item[{\rm (a)}] If $\mathscr{C}(\mathcal{A}_{\bm a},E)\supseteq\mathscr{C}(\mathcal{A}_{\bm a'},E)$, then 
\[
w_i\big(\mathcal{C}(\mathcal{A}_{\bm a},E)\big)\le w_i\big(\mathcal{C}(\mathcal{A}_{\bm a'},E)\big)\quad\text{for}\quad i=0,1,\ldots,r(\mathcal{A}_{\bm o}).
\]
\item[{\rm (b)}] If  $\alpha_{\bm a}(C)\le\alpha_{\bm a'}(C)$ for all circuits $C$ of $M_{\mathcal{A}_{\bm o}}$, then
\[
w_i(M_{\mathcal{A}_{\bm o}},\alpha_{\bm a})\le w_i(M_{\mathcal{A}_{\bm o}},\alpha_{\bm a'})\quad\text{for}\quad i=0,1,\ldots,r(M_{\mathcal{A}_{\bm o}}).
\]
\end{itemize}
\end{corollary}
\begin{proof}
By the same argument as in the proof of \autoref{Comparison}, we derive that 
\[
w_i\big(\mathcal{C}(\mathcal{A}_{\bm a},E)\big)\le w_i\big(\mathcal{C}(\mathcal{A}_{\bm a'},E)\big).
\]
Note that for any $\bm a,\bm a'\in\mathbb{F}^{|E|}$, 
\[
\mathscr{C}(\mathcal{A}_{\bm a},E)\supseteq \mathscr{C}(\mathcal{A}_{\bm a'},E)\Longleftrightarrow\mathscr{C}(\mathcal{A}_{\bm a})\supseteq\mathscr{C}(\mathcal{A}_{\bm a'})\Longleftrightarrow \alpha_{\bm a}(C)\le\alpha_{\bm a'}(C),\;\forall\, C\in \mathscr{C}(\mathcal{A}_{\bm o}).
\]
It follows from part (b) of \autoref{Relation-Compatible-Semimatroid1} that the second part holds.
\end{proof}

Using assignings of $M_{\mathcal{A}_{\bm o}}$, we now restate the unified comparison relations among unsigned coefficients of the characteristic polynomials for parallel translations, established in \cite[Theorem 1.2]{CFW2021}. Let $\bm a\in\mathbb{F}^{|E|}$. Write 
\[
 \chi(\mathcal{A}_{\bm a};t)=w_0(\mathcal{A}_{\bm a})t^n- w_1(\mathcal{A}_{\bm a})t^{n-1}+\cdots+(-1)^{n-r(\mathcal{A}_{\bm a})} w_{r(\mathcal{A}_{\bm a})}(\mathcal{A}_{\bm a}).
\]
Immediately, \autoref{Comparison1} directly yields \cite[Theorem 1.2]{CFW2021}.
\begin{proposition}[\cite{CFW2021}, Theorem 1.2]
Let $\bm a,\bm a'\in\mathbb{F}^{|E|}$. If $\alpha_{\bm a}(C)\le\alpha_{\bm a'}(C)$ for all circuits $C$ of $M_{\mathcal{A}_{\bm o}}$, then
\[ 
w_i(\mathcal{A}_{\bm a})\le w_i(\mathcal{A}_{\bm a'}),\quad i=0,1,\ldots,r(\mathcal{A}_{\bm o}).
\]
\end{proposition}

Furthermore, analogous to matroid,  we consider the poset $\mathscr{L}(M,\alpha)$, which consists of compatible subsets $X$ of $E(M)$ such that no other element $Y\in\mathcal{C}(M,\alpha)$ contains $X$ with rank $r_{M}(X)$, ordered by set inclusion. 
Particularly, $\mathscr{L}(M_{\mathcal{A}_{\bm o}},\alpha_{\bm a})=\mathscr{L}\big(\mathcal{C}(\mathcal{A}_{\bm a},E)\big)$ follows from \autoref{Semimatroid-Inverse}. We end this section with an alternative expression for $\chi(M_{\mathcal{A}_{\bm o}},\alpha_{\bm a};t)$ by the M\"obius function of $\mathscr{L}(M_{\mathcal{A}_{\bm o}},\alpha_{\bm a})$, and with the relation between $\chi(\mathcal{A}_{\bm a};t)$ and $\chi(M_{\mathcal{A}_{\bm o}},\alpha_{\bm a};t)$. 
\begin{corollary}\label{Characteristic-Arrangement-Matroid}
Let $\mathcal{A}=\{H_e:\bm\alpha_e\cdot\bm x=a_e\mid e\in E\}$ be a hyperplane arrangement in $\mathbb{F}^n$ and $\bm a=(a_e)_{e\in E}\in\mathbb{F}^{|E|}$. Then, 
\begin{itemize}
\item[{\rm (a)}]$
\chi(M_{\mathcal{A}_{\bm o}},\alpha_{\bm a};t)=\sum_{X\in\mathscr{L}(M_{\mathcal{A}_{\bm o}},\alpha_{\bm a})}\mu(X)t^{r(M_{\mathcal{A}_{\bm o}})-r_{M_{\mathcal{A}_{\bm o}}}(X)}.
$
\item[{\rm (b)}]$\chi(\mathcal{A};t)=t^{n-r(\mathcal{A})}\chi(M_{\mathcal{A}_{\bm o}},\alpha_{\bm a};t).$
\end{itemize}
\end{corollary}
\begin{proof}
Part (a) is an immediate result of \autoref{Mobius-Characteristic-Semi}. By integrating the two relations: 
\[
\chi(M_{\mathcal{A}_{\bm o}},\alpha_{\bm a};t)=\chi\big(\mathcal{C}(\mathcal{A},E);t\big)
\]
from part (b) of \autoref{Relation-Compatible-Semimatroid1} and
\[
\chi(\mathcal{A};t)=t^{n-r(\mathcal{A})}\chi\big(\mathcal{C}(\mathcal{A},E);t\big)
\] 
from \eqref{Semimatroid-Arrangement}, we deduce $\chi(\mathcal{A};t)=t^{n-r(\mathcal{A})}\chi(M_{\mathcal{A}_{\bm o}},\alpha_{\bm a};t)$. 
\end{proof}

\section{Assigning graphs}\label{Sec6}
In this section, our primary objective is to explore applications of assigning matroids to assigning graphs, with a particular focus on their polynomial invariants and graph colorings. 
\subsection{Compatible chromatic polynomials of assigning graphs}\label{Sec6-1}
This subsection aims to define an assigning graph and its compatible chromatic polynomial, and further to establish a close link between the compatible chromatic polynomial of assigning graphs and the compatible characteristic polynomial of their associated cycle matroids.

Graph terminology can refer to Bondy's book  \cite{Bondy2008}. Let $G=\big(V(G),E(G)\big)$ be a finite graph with vertex set $V(G)$ and edge set $E(G)$, allowing multiple edges and loops. An {\em orientation} of a graph is an assignment of direction to each edge. In this context, each oriented edge is referred to as an {\em arc}. An edge with identical ends is called a {\em loop}, and with distinct ends is a {\em link}. A graph $H$ is a {\em subgraph} of $G$ if $V(H)\subseteq V(G)$ and $E(H)\subseteq E(G)$. When $V(H)=V(G)$, the subgraph $H$ of $G$ is referred to as a {\em spanning subgraph}. Let $X\subseteq E(G)$. The subgraph $G-X$ is obtained from $G$ by removing all edges in $X$. Denote by $G|X$ the spanning subgraph with edges $X$, that is, $G|X:=G-\big(E(G)-X\big)$. The graph $G/X$ is obtained from $G$ by contracting all edges in $X$. The {\em rank function} $r_G$ of $G$ is given by that for any spanning subgraph $H$ of $G$,
\[
r_G(H):=|V(G)|-c(H),
\]
where $c(H)$ is the number of (connected) components of $H$. Throughout this section, we assume that the notation $S^E$ represents the collection of mappings from a finite set $E$ to a set $S$, unless otherwise stated. Members in $S^E$ are considered as vectors indexed by $E$.
  
Graph coloring is a significant subfield in graph theory that originated in the middle of the 19th century with the famous Four Color Problem. A {\em$k$-vertex coloring}, or simply a {\em$k$-coloring}, of $G$ is a mapping $c:V(G)\to S$, where $S$ is a set of distinct $k$ colors, typically $S=\{1, 2, . . . , k\}$. A vertex coloring is said to be {\em proper} if no two adjacent vertices share the same color. More generally, Jaeger, Linial, Payan and Tarsi \cite{JLPT1992} generalized the concept to the group coloring in 1992, which is the dual of group connectivity in graphs. Group connectivity and group colorings of graphs are nicely surveyed in \cite{LSZ2011}. Consider a finite additive Abelian group $A$ with the identity element $0$ as the coloring set. 
Let $ f: E(G)\to A$ be an edge-value function. A vertex coloring $c:V(G)\to A$ is called an {\em$(A, f)$-coloring} if $c(v)-c(u)\ne f(e)$ for each arc $D(e):=uv$ directed from $u$ to $v$ with respect to the orientation $D$. Naturally, the $(A,0)$-coloring agrees with the ordinary proper coloring. 

It is worth noting that there is a close link between $(A, f)$-colorings and zero-free colorations of gain graphs that often goes unnoticed. Specifically, $(A, f)$-colorings essentially correspond to zero-free colorations of gain graphs described in \cite[Section III.4]{Zaslavsky1995}. Consequently, the balanced chromatic polynomial of gain graphs, as introduced by Zaslavsky \cite{Zaslavsky1995}, enumerates the $(A,f)$-colorings. In this section, we introduce compatible chromatic polynomials of an assigning graph counting the $(A, f)$-colorings associated with zero-one assignings of its cycles  (For more details, please refer to see \autoref{Sec7-3}).

A {\em cycle } $C$ of $G$ is a connected 2-regular subgraph of $G$. Its corresponding edge set $E(C)$ is referred to as a {\em circuit} of $G$. Let $\mathscr{C}(G)$ be the collection of cycles of $G$. An {\em assigning} $\alpha$ of $G$ is a mapping from $\mathscr{C}(G)$ to the set $\{0,1\}$. Especially, we define $\alpha\equiv0$ if $G$ contains no cycles, i.e., $\mathscr{C}(G)=\emptyset$. 

With an arbitrary assigning $\alpha$ of $G$, we call the pair $(G,\alpha)$ an {\em assigning graph}. Denote by 
\[
\mathscr{C}(G,\alpha):=\big\{C\in\mathscr{C}(G)\mid \alpha(C)=0\big\},
\] 
which is called the {\em compatible cycle set} of $(G,\alpha)$. A cycle in $\mathscr{C}(G,\alpha)$, and more broadly any subgraph of $G$ whose cycles are in $\mathscr{C}(G,\alpha)$, is {\em $\alpha$-compatible} (simply {\em compatible}) and {\em incompatible} otherwise. Let $\mathcal{C}(G,\alpha)$ be the set of all compatible spanning subgraphs of $(G,\alpha)$.  

The {\em compatible chromatic polynomial} of an assigning graph $(G,\alpha)$ is defined by
\[
\chi(G,\alpha;t):=\sum_{H\in\mathcal{C}(G,\alpha)}(-1)^{|E(H)|}t^{c(H)}.
\] 
In particular, $\chi(G,0;t)$ coincides with the classical chromatic polynomial in \cite{Whitney1932,Whitney1932-1}. 

In the upcoming theorem, we describe a connection between the compatible chromatic polynomial of assigning graphs and the compatible characteristic polynomial of their assigning cycle matroids. The {\em cycle matroid} $M_G$ of $G$ has the edge set $E(G)$ as its ground set, and the rank function $r_{M_G}$ is given by
\[
r_{M_G}(X):=r_G(G|X),\quad\forall\, X\subseteq E(G),
\]
with circuits corresponding to edge sets of all cycles  (see \cite[Proposition 1.1.7]{Oxley2011}). Associated with any assigning $\alpha$ of $G$, we define the assigning $\alpha_{M_G}$ of $M_G$ by that
\[
\alpha_{M_G}\big(E(C)\big):=\alpha(C), \quad\forall\,C\in\mathscr{C}(G).
\]
\begin{theorem}\label{Semimatroid-Graph}
Let $(G,\alpha)$ be an assigning graph. Then
\[
\chi(G,\alpha;t)=t^{c(G)}\chi(M_G,\alpha_{M_G};t).
\]
\end{theorem}
\begin{proof}
According to $\alpha_{M_G}\big(E(C)\big)=\alpha(C)$ for any $C\in\mathscr{C}(G,\alpha)$, there exists a natural one-to-one correspondence between $\mathcal{C}(G,\alpha)$ and $\mathcal{C}(M_G,\alpha_{M_G})$, with every $H\in \mathcal{C}(G,\alpha)$ corresponding to its edge set $E(H)\in \mathcal{C}(M_G,\alpha_{M_G})$. By $r_G(G|X)=|V(G)|-c(G|X)$ and $r_{M_G}(X)=r_G(G|X)$ for all $X\subseteq E(G)$, we arrive at 
\begin{align*}
\chi(G,\alpha;t)&=\sum_{X\in\mathcal{C}(M_G,\alpha_{M_G})}(-1)^{|X|}t^{c(G|X)}\\
&=t^{c(G)}\sum_{X\in\mathcal{C}(M_G,\alpha_{M_G})}(-1)^{|X|}t^{r(M_G)-r_{M_G}(X)}\\
&=t^{c(G)}\chi(M_G,\alpha_{M_G};t).
\end{align*}
This completes the proof.
\end{proof}
\subsection{Colorings}\label{Sec7-3}
In this subsection, we revisit general graph colorings from the perspective of assigning matroids and hyperplane arrangements.

For convenience, we shall work with a fixed field $\mathbb{F}$ as the coloring set and a fixed reference orientation $D$ of $G$. Let $D(G)$ represent the set of all arcs of $G$ with respect to $D$. For simplicity, we assume the vertex set $V(G)$ of $G$ to be the set $\{1,2,\ldots,n\}$. We denote the arc from $i$ to $j$ as  $e_{ij}$ and its opposite orientation  $-e_{ij}$. Associated with the orientation $D$,  a {\em graphic arrangement} $\mathcal{A}_G$ ($\mathcal{A}_G$ may be a multi-arrangement) in $\mathbb{F}^n$ is defined as
\[
\mathcal{A}_G:=\big\{H_{e_{ij}}: x_j-x_i=0\mid e_{ij}\in D(G)\big\}.
\]
We need to be especially careful when $G$ contains a loop $e_{ii}\in D(G)$. Each such loop $e_{ii}$ corresponds precisely to the degenerate hyperplane $H_{e_{ii}}:x_i-x_i=0$, which is exactly the whole space $\mathbb{F}^n$. In this case, the characteristic polynomial $\chi(\mathcal{A}_G;t)$ is equal to zero. 

Given a vector or a function $\bm a=(a_{ij})_{e_{ij}\in D(G)}\in\mathbb{F}^{|D(G)|}$, we now consider parallel translations $\mathcal{A}_{G,\bm a}$ of the graphic arrangement $\mathcal{A}_G$ as follows:
\[
\mathcal{A}_{G,\bm a}:=\big\{H_{e_{ij},a_{ij}}: x_j-x_i=a_{ij}\mid e_{ij}\in D(G)\big\},
\]
which is called the {\em affinographic hyperplane arrangement} by Forge and Zaslavsky \cite{Zaslavsky2003}. Remarkable, the matroid $M_{\mathcal{A}_G}$ given by $\mathcal{A}_G$ coincides precisely with the cycle matroid $M_G$ defined by $G$. Recall from \eqref{Assigning-Arrangement} that every parallel translation $\mathcal{A}_{G,\bm a}$ naturally defines an assigning $\alpha_{\bm a}$ of $G$ as follows: for any cycle $C$ of $G$,
\[
\alpha_{\bm a}(C):=
\begin{cases}
0,&\text{if $E(C)\in\mathscr{C}(\mathcal{A}_{G,\bm a})$};\\
1,&\text{otherwise}.
\end{cases}
\]
Such an assigning $\alpha_{\bm a}$ is known as an {\em admissible assigning} of $G$. In this case, we call the pair $(G,\alpha_{\bm a})$ an {\em admissible assigning graph}. 

Furthermore, we define the poset $\mathscr{L}(G,\alpha)$ as the collection of compatible spanning subgraphs $H$ of $(G,\alpha)$ such that there is no other compatible spanning subgraph $H'$ with rank $r_G(H)$ satisfying $E(H)\subseteq E(H')$, ordered by edge set inclusion. The following result shows that the $(\mathbb{F},\bm a)$-colorings can be geometrically realized through the parallel translation $\mathcal{A}_{G,\bm a}$ of the graphic arrangement $\mathcal{A}_G$, and also provides an alternative expression for the compatible chromatic polynomial $\chi(G,\alpha_{\bm a};t)$ using the M\"obius function of $\mathscr{L}(G,\alpha_{\bm a})$. 
\begin{proposition}\label{Arrangement-Graph-Chromatic}
Let $\bm a=(a_{ij})_{e_{ij}\in D(G)}\in\mathbb{F}^{|D(G)|}$. Then,
\begin{itemize}
\item [{\rm(a)}] $\chi(G,\alpha_{\bm a};t)=\chi(\mathcal{A}_{\bm a};t)$.
\item [{\rm(b)}] $\chi(G,\alpha_{\bm a};t)=\sum_{H\in\mathscr{L}(G,\alpha_{\bm{a}})}\mu(X)t^{c(H)}$.
\item [{\rm(c)}]Let $p$ be a large enough prime number. If $\mathbb{F}$ is the finite field $\mathbb{F}_q$ of $q=p^r$ elements, then the number of $(\mathbb{F}_q,\bm a)$-colorings of $G$ equals the number of points in $M(\mathcal{A}_{G,\bm a})$.
\end{itemize}
\end{proposition}
\begin{proof}
Parts (a) and (b) are established by \autoref{Characteristic-Arrangement-Matroid} and \autoref{Semimatroid-Graph}. Together with \cite[Theorem 2.2]{Athanasiadis1996} that $\chi(\mathcal{A}_{G,\bm a},q)$ is the number of points in the complement $M(\mathcal{A}_{G,\bm a})$,  we arrive at that part (c) holds.
\end{proof}

It should be noted that an admissible assigning graph $(G,\alpha_{\bm a})$ corresponds to a gain graph $(G,\bm a)$ that was first introduced by Zaslavsky in \cite{Zaslavsky1989}. In this context, the compatible chromatic polynomial aligns with the balanced chromatic polynomial proposed by Zaslavsky in \cite{Zaslavsky1995} (see also \cite[(2.1)]{FZ2007}). Therefore, part (a) of \autoref{Arrangement-Graph-Chromatic} can be derived from \cite[Theorem III.5.2]{Zaslavsky1995} and \cite[Corollary IV.4.5]{Zaslavsky2003}; and part (b) of \autoref{Arrangement-Graph-Chromatic} is adaptation of results in \cite[Sections III.4 and III.5]{Zaslavsky1995} (see also \cite[(2.2)]{FZ2007}).
\section*{Acknowledgements}
The work is supported by National Natural Science Foundation of China (12301424).

\end{document}